\documentclass{elsarticle}

\usepackage{amssymb}
\usepackage{amsthm}
\usepackage{amsmath,amsfonts}
\usepackage[dvips]{epsfig}
\usepackage{tikz}
\usetikzlibrary{arrows.meta}
\usetikzlibrary{arrows}
\usepackage{pgfplots}
\usepackage{url}
\usepackage{subcaption}
\usepackage{booktabs}
\usepackage{array}
\usepackage{graphicx}
\usepackage{bm}
\usepackage{xcolor}
\newtheorem{problem}{Problem}
\newtheorem{notation}{Notation}
\newtheorem{remark}{Remark}
\newtheorem{theorem}{Theorem}
\newtheorem{lemma}[theorem]{Lemma}

\newenvironment{proofof}[1]{{\bf Proof #1.}}{\hfill$\square$}
\newcommand{\kbox}[1]{ {#1^{(k)}}\raisebox{.8em}{}}
\newcommand{\kkbox}[2]{ {#1^{(#2)}}\raisebox{.8em}{}}
\newcommand{\kappaess}{\kappa_{\mathrm{ess}}}
\DeclareMathAlphabet{\mathpzc}{OT1}{pzc}{m}{it}
\newcommand{\degree}{\mathpzc{p}}
\newcommand{\curl}{\text{\normalfont curl}\,}
\newcommand{\divergence}{\text{\normalfont div}\,}
\newcommand{\I}{\mathrm I}

\newcommand{\scaling}{\mathcal D}
\journal{arXiv}
\begin{document}
\begin{frontmatter}
\title{Isogeometric Tearing and Interconnecting Solvers for Linearized Elasticity in multi-patch Isogeometric Analysis}

\author[js]{Jarle Sogn}
\address[js]{Department of Mathematics, University of Oslo, Postboks 1053, Blindern, 0316 Oslo, Norway}
\author[st]{Stefan Takacs}
\address[st]{Institute of Numerical Mathematics, Johannes Kepler University Linz, Altenberger~Str.~69, 4040 Linz, Austria}
\begin{abstract}
We consider the linearized elasticity equation, discretized with multi-patch Isogeometric Analysis. A standard discretization error analysis is based on Korn's inequality, which degrades for certain geometries, such as long and thin cantilevers. This phenomenon is known as geometry locking. We observe that high-order methods, like Isogeometric Analysis is beneficial in such a setting.
The main focus of this paper is the construction and analysis of a domain decomposition solver, namely an Isogeometric Tearing and Interconnecting (IETI) solver, where we prove that the convergence behavior does not depend on the constant of Korn's inequality for the overall domain, but only on the corresponding constants for the individual patches. Moreover, our analysis is explicit in the choice of the spline degree. Numerical experiments are provided which demonstrates the efficiency of the proposed solver.
\end{abstract}
\end{frontmatter}

\section{Introduction}

In this paper, we consider a discretization and a solution strategy for the linearized elasticity problem. For small deformations $\mathbf{u}$ of a body in static equilibrium, the equation
\begin{equation}
  \label{eq:primalform}
- 2\mu \; \divergence \epsilon(\mathbf{u}) -\lambda \; \nabla \divergence \mathbf{u}   = \mathbf{f}
\end{equation}
holds in a variational sense, where $\epsilon(\mathbf u) = \frac12 (\nabla \mathbf u + \nabla^\top \mathbf u)$ is the symmetric gradient, $\lambda,\mu>0$ are Lam\'e coefficients modelling the material properties and $\mathbf{f}$ is the external volumetric force acting on the body, cf.~\cite{gould1994introduction}. In this model, the computational domain $\Omega \subset \mathbb R^d$ models the shape of the undeformed object with boundary $\partial \Omega$.  As boundary conditions, we use Dirichlet boundary conditions to prescribe the displacement on some part $\Gamma_D$ of the boundary with positive measure and Neumann boundary conditions on the  remainder $\Gamma_N := \partial\Omega\setminus\Gamma_D$ of the boundary to prescribe acting forces:
\begin{align}
  \mathbf u &= \mathbf g_D & \mbox{on}\quad \Gamma_D, \label{eq:dirichlet} \\
  2\mu\; \epsilon(\mathbf u)\; \mathbf n+\lambda \; \divergence \mathbf u \; \mathbf n &= \mathbf g_N &\mbox{on}\quad \Gamma_N, \label{eq:neumann}
\end{align}
where $\mathbf g_D$ denotes the prescribed displacement and $\mathbf g_N$ the prescribed boundary force density.

Several challenges may arise when discretizing and solving Equation~\eqref{eq:primalform}. The ratio $\mu/\lambda$ takes small values for nearly incomressible materials. The limit case $\mu/\lambda\to0$ can be understood as the case of completely incompressible materials. In these cases, the well-posedness (in the sense of the Lax-Milgram lemma) deteriorates, which might lead to a poor approximation of the solution with standard Galerkin discretizations. This phenomenon is known as material locking and geometry locking. A common remedy is to introduce an auxiliary variable and derive a mixed variational formulation, cf.~\cite{boffi2013mixed, stenberg1986construction, arnold2007mixed, auricchio2007fully}. The focus of this paper is not on material locking; thus, we assume $\mu/\lambda$ is well bounded away from zero.

Another challenge is related to the symmetric gradient. For showing well-posedness in the standard Sobolev space $H^1$, one typically uses Korn's inequality, which introduces a constant that usually depends on the geometry of the object in question. This happens for example for long and thin cantilevers. Also such domains might lead to a poor approximation of the solution with standard Galerkin discretizations, which is known as geometry locking. We will discuss this issue in detail and observe that high-order methods help to mitigate related issues, compared to standard discretizations with piecewise linear functions.

We use Isogeometric Analysis (IgA) as discretization method. IgA was introduced in the seminal paper~\cite{hughes2005isogeometric} as a new technique for discretizing partial differential equations (PDEs) that improves the integration of simulation and computer aided design (CAD). In IgA, both the computational domain and the solution of the PDE are represented as linear combination of tensor-product B-splines or non-uniform rational B-splines (NURBS). Only simple domains can be parameterized using a single geometry mapping. More complicated domains are usually composed of several patches, where each patch is parameterized independently (multi-patch domains). For more information on IgA, see the review paper \cite{da2014mathematical} and the references therein.

The main focus of this paper is the intoduction and the analysis of a fast domain decomposition solver for the linearized elasticity problem, discretized using IgA on multi-patch domains. We consider FETI-DP methods, which were originally introduced in~\cite{farhat2001feti} and has been used to solve linear elasticity, see, e.g.,~\cite{KLAWONN20071400,farhat2000scalable,KlawonnWidlund:2006}. For multi-patch IgA domains, the patches themselves serve as a natural choice for the substructures. FETI-DP was first adapted to IgA in~\cite{kleiss2012ieti} and named the \textit{Dual-Primal Isogeometric Tearing and Interconnecting} (IETI-DP) method. In recent years, there has been a lot of research in IETI-DP methods, see, e.g.,~\cite{hofer2017dual, SchneckenleitnerTakacs:2020, SognTakacs:2022, widlund2021block, schneckenleitner2022ieti, montardini2022ieti}.

Isogeometric domain decomposition methods for mixed linear elasticity systems have been studied in \cite{widlund2021block,pavarino2016isogeometric}, where convergence estimates are provided. We study the primal formulation and use the results from \cite{MandelDohrmannTezaur:2005a, SchneckenleitnerTakacs:2020} to prove a convergence estimate which is explicit with respect to both the spline degree and grid-size (Theorem~\ref{thrm:fin}). The analysis shows that the performance of the solver does not depend on the shape of the overall computational domain, but only on the shapes of the individual patches. This means that one can obtain a robust IETI solver by setting up an adequate subdivision of the computational domain into patches.

This paper is organized as follows. In Section~\ref{sec:modelProblem}, we introduce the variational formulation of the model problem and its discretization. The well-posedness and the discretization errors are discussed in the subsequent Section~\ref{sec:approx}. The proposed solver is introduced in Section~\ref{sec:IETISolver}. The following Section~\ref{sec:ieti:theory} is dedicated to its analysis. We close with numerical results and concluding remarks in Section~\ref{sec:num}.

\section{Problem formulation and discretization}
\label{sec:modelProblem}

In this section we formulate the variational formulation of the elasticity problem~\eqref{eq:primalform}~--~\eqref{eq:neumann} and its discretization.
Let $\Omega\subset\mathbb{R}^2$ be an open, connected and bounded domain with Lipschitz boundary $\partial \Omega$. $L^2(\Omega)$ and $H^s(\Omega)$ denote the standard Lebesgue and Sobolev spaces with standard norms $\|\cdot\|_{L^2(\Omega)}$, $\|\cdot\|_{H^s(\Omega)}$, seminorms $|\cdot|_{H^s(\Omega)}$, and scalar products $(\cdot,\cdot)_{L^2(\Omega)}$. Analogously, we define these spaces, norms, seminorms and scalar products on other domains and on manifolds in $\mathbb R^2$, like boundaries. For ease of notation, we assume that Dirichlet boundary conditions are homogeneous. So, the displacement has to be in $H^1_{0,D}(\Omega) := \left\{ v\in H^1(\Omega) : v|_{\Gamma_D} = 0 \right\}$. Here and in what follows a restriction to (a part of) the boundary is to be understood in the sense of a trace operator. By standard arguments, the following problem is the weak form of the elasticity equations.

\begin{problem}\label{prob1}
	Given $\mathbf f\in [L^2(\Omega)]^2$ and $\mathbf g_N\in [L^2(\Gamma_N)]^2$,
	find $\mathbf u\in [H^1_{0,D}(\Omega)]^2$ such that
	\begin{equation}
	  \label{eq:ellipt:elasticity}
	 a(\mathbf u,\mathbf v)
	= \langle F, \mathbf{v}\rangle \quad \mbox{for all}\quad \mathbf v \in [H^1_{0,D}(\Omega)]^2,
	\end{equation}
	where
	\begin{align*}
		a(\mathbf u,\mathbf v)
		&:=
		2\mu (\epsilon(\mathbf{u}), \epsilon(\mathbf{v}))_{L^2(\Omega)}
	+\lambda\; ( \divergence \mathbf{u},\divergence \mathbf{v})_{L^2(\Omega)},
	\mbox{ and }
	\\
	\langle F, \mathbf{v} \rangle
	&:=
	( \mathbf{f},\mathbf{v} )_{L^2(\Omega)}	+  (\mathbf{g}_N, \mathbf{v} )_{L^2(\Gamma_N)}.
	\end{align*}
\end{problem}

We discretize the elasticity equations by means of multi-patch Isogeometric Analysis.
We assume that the computational domain $\Omega$ is composed
of $K$ non-overlapping patches $\Omega^{(k)}$, i.e., the domains
$\Omega^{(k)}$ are open, connected and bounded with Lipschitz boundary such that
\begin{align*}
	\overline{\Omega} = \bigcup_{k=1}^K \overline{\Omega^{(k)}} \quad \text{and}\quad
	\Omega^{(k)} \cap \Omega^{(\ell)} = \emptyset \quad\text{for all}\quad k \neq \ell,
\end{align*}
where $\overline{T}$ denotes the closure of the set $T$.
We require that the patches form an admissible decomposition, i.e.,
that there are no T-junctions. This means that for any two patch indices $k\not=\ell$, the set	$\partial{\Omega^{(k)}} \cap \partial{\Omega^{(\ell)}}$	is either a common edge
$\Gamma^{(k,\ell)}:=\partial{\Omega^{(k)}} \cap \partial{\Omega^{(\ell)}}$,
a common vertex, or empty (cf.~\cite[Ass.~2]{SchneckenleitnerTakacs:2020}).
Moreover, we assume that both the Neumann boundary and the Dirichlet boundary consist of whole edges.

For each patch index $k$, the set $\mathcal N_{\Gamma}(k)$ contains the indices $\ell$ of
patches $\Omega^{(\ell)}$ that share an edge with $\Omega^{(k)}$.
The common vertices of two or more
patches -- that are not located on the Dirichlet boundary -- are denoted by
$\mathbf x_1,\ldots,\mathbf x_J$. For each $j=1,\ldots,J$, the set $\mathcal{N}_{\mathbf x}(j)$ contains the
indices of all patches $\Omega^{(k)}$ such that $x_j\in\partial\Omega^{(k)}$.
We assume that there is a constant $C_1$ such that
\begin{equation}\label{eq:ass:neighbor}
		|\mathcal{N}_{\mathbf x}(j)|\le C_1
				\quad \mbox{for all}\quad j=1,\ldots,J
\end{equation}
(cf. \cite[Ass.~3]{SchneckenleitnerTakacs:2020}).
The set of all vertices associated to one patch is denoted by
\[
	\mathcal{N}_{\mathbf x}^{-1}(k) := \{ j : k\in \mathcal{N}_{\mathbf x}(j) \}.
\]
Each patch $\Omega^{(k)}$ is parameterized by a geometry mapping
\begin{align}
	\mathbf{G}_k:\widehat{\Omega}:=(0,1)^2 \rightarrow \Omega^{(k)}:=\mathbf{G}_k(\widehat{\Omega}) \subset \mathbb{R}^2,
\end{align}
which can be continuously extended to the closure of the parameter domain
$\widehat{\Omega}$. In IgA, the geometry mapping is typically represented using B-splines or NURBS. As usual, the computational methods do not depend on such a representation. We only assume that the geometry mappings are not too much distorted, i.e., there is a constant $C_2>0$ such that
\begin{equation}\label{eq:ass:nabla}
		\| \nabla \mathbf G_k \|_{L^\infty(\widehat{\Omega})} \le C_2\, H_k
		\quad\text{and}\quad
		\| (\nabla \mathbf G_k)^{-1} \|_{L^\infty(\widehat{\Omega})} \le C_2\, H_k^{-1}
\end{equation}
holds for all $k=1,\ldots,K$, where $H_k$ is the diameter of the patch $\Omega^{(k)}$
(cf.~\cite[Ass.~1]{SchneckenleitnerTakacs:2020}).

On the parameter domain $\widehat \Omega=(0,1)^2$, we choose a B-spline space of some freely chosen spline degree $\degree\in\mathbb N:=\{1,2,3,\ldots\}$. The spline space depends on a freely
chosen $\degree$-open knot vector
$\Xi^{(k,\delta)}:=(\xi^{(k,\delta)}_0,\cdots,\xi^{(k,\delta)}_{N^{(k,\delta)}})$ for each patch $k$ and each spacial direction $\delta\in\{1,2\}$. We allow repeated (inner) knots in order to reduce smoothness; since we require a $H^1$-conforming discretization, no more than $\degree$ interior knots may be repeated. Based on these knot vectors, the corresponding B-splines are obtained by the Cox-de Boor formula (cf.~\cite[(2.1) and (2.2)]{Cottrell:Hughes:Bazilevs}); these B-splines span the spline space $S^{(k,\delta,\degree)}\subset H^1(0,1)$. Based on these univariate splines, we introduce the corresponding tensor product spline space as
$S^{(k,\degree)} := S^{(k,1,\degree)} \otimes S^{(k,2,\degree)}$
as discretization space on the parameter domain $\widehat\Omega$, equipped with the standard tensor product basis.

The function spaces on the physical patches $\Omega^{(k)}$ are defined via
the \emph{pull-back principle}:
\[
		\mathbf V^{(k)} := \left\{ v : v\circ \mathbf G_k \in S^{(k,\degree)},\; v\big|_{\Gamma_D\cap\partial\Omega^{(k)}} = 0\right\}^2.
\]
The grid size $\widehat h_k$ on the parameter domain and the grid size $h_k$ on the physical patch are defined by
\begin{equation}\label{eq:gridsize}
		\widehat h_k:=\max\left\{ \xi_i^{(k,\delta)}-\xi_{i-1}^{(k,\delta)}
								\,:\, i=1,\ldots,N^{(k,\delta)}  ,\, \delta=1,2 \right\}
								\quad\mbox{and}\quad
								h_k:=H_k \widehat h_k,
\end{equation}
where the definition of the latter is motivated by~\eqref{eq:ass:nabla}. Analogously, we
define the minimum grid size
\[
		\widehat h_{\min,k}:=\min\left\{ \xi_i^{(k,\delta)}-\xi_{i-1}^{(k,\delta)}\ne0
								\,:\, i=1,\ldots,N^{(k,\delta)}  ,\, \delta=1,2 \right\}
\]
and $h_{\min,k}:=H_{k} \widehat h_{\min,k}$.
We assume the grid to be quasi-uniformity, i.e., there is a constant $C_3$ such that
\begin{equation}\label{eq:ass:quasiuniform}
		 \frac{\widehat h_{k}}{\widehat h_{\min,k}}=\frac{h_{k}}{h_{\min,k}}\le C_3 \text{ for } k=1,\ldots,K,
\end{equation}
cf.~\cite[Ass.~4]{SchneckenleitnerTakacs:2020}.

Finally, we introduce a global discretization space. We assume that the discretization is fully matching (see \cite[Ass.~5]{SchneckenleitnerTakacs:2020}), this means that for any two patches $\Omega^{(k)}$ and $\Omega^{(\ell)}$, sharing an edge $\Gamma^{(k,\ell)}$, the traces agree, $\mathbf V^{(k)}\big|_{\Gamma^{(k,\ell)}}=\mathbf V^{(\ell)}\big|_{\Gamma^{(k,\ell)}}$ and are represented using the same basis. So, as global discretization space, we use
\[
		\mathbf{V} = \left\lbrace \mathbf{v}\in [H^1_{0,D} (\Omega)]^2
		\,:\,
		\mathbf{v}|_{\Omega^{(k)}}\in \mathbf{V}^{(k)} \text{ for } k=1,
		\ldots,K\right\rbrace.
\]
For non-matching interfaces or for patches with T-junctions, discontinuous Galerkin methods could be considered, see, e.g.,~\cite{schneckenleitner2022ieti}. This, however goes beyond this paper.

The discretized problem reads as follows.
\begin{problem}\label{prob2}
	Given $\mathbf f\in [L^2(\Omega)]^2$ and $\mathbf g_N\in [L^2(\Gamma_N)]^2$,
	find $\mathbf u_h\in \mathbf V$ such that
	\begin{equation}
		\label{eq:discelast}
		a(\mathbf u_h,\mathbf v_h)
		= \langle F, \mathbf{v}_h\rangle \quad \mbox{for all}\quad \mathbf v_h \in \mathbf V,
	\end{equation}
	where $a(\cdot,\cdot)$ and $\langle F,\cdot\rangle$ are as in Problem~\ref{prob1}.
\end{problem}

\begin{notation}\label{notation}
		Here and in what follows, $c$ denotes a generic positive constant that only
		depends on the constants $C_1$, $C_2$ and $C_3$ from \eqref{eq:ass:neighbor}, \eqref{eq:ass:nabla} and \eqref{eq:ass:quasiuniform}. We write $a \lesssim b$ if there is such a generic constant such that $a\le c\, b$. We write $a\eqsim b$ if $a\lesssim b\lesssim a$.
\end{notation}

\section{Well posedness and discretization error}\label{sec:approx}

To show well-posedness both of Problem~\ref{prob1} and the discretized Problem~\ref{prob2},  we use the Lax-Milgram lemma~\cite[Theorem~2.7.7]{BrennerScott}, which is applicable since both $[H^1_{0,D}(\Omega)]^2$ and $\mathbf{V}$ are Hilbert spaces.
Thus, existence and uniqueness of solutions to the Problems~\ref{prob1} and~\ref{prob2} and their continuous dependence on the data is guaranteed if $a(\cdot,\cdot)$ is coercive and bounded and if $\langle F,\cdot \rangle$ is bounded; the latter follows directly from the Cauchy-Schwarz inequality and the trace theorem. For the further analysis, the estimates on the bilinear form are more relevant; also here, the boundedness is straight forward:
\begin{equation}\label{eq:a:bounded}
	a(\textbf u,\textbf v) \le (2\mu+2\lambda) \|\epsilon(\mathbf{u})\|_{L^2(\Omega)}\|\epsilon(\mathbf{v})\|_{L^2(\Omega)}
	\le (2\mu+2\lambda) |\mathbf{u}|_{H^1(\Omega)}|\mathbf{v}|_{H^1(\Omega)}
	.
\end{equation}
The crucial point is the coercivity. Assuming
\begin{equation}
		\label{eq:korn}
		\alpha_{\Omega,\Gamma_D} := \inf_{\mathbf u\in H^1_{0,D}(\Omega)}
		\frac{\|\epsilon(\mathbf u)\|_{L^2(\Omega)}}
		{|\mathbf u|_{H^1(\Omega)}} > 0,
\end{equation}
we obtain
\begin{equation*}
		a(\textbf u,\textbf u)
		\ge
		2\mu\|\epsilon( \mathbf{u})\|_{L^2(\Omega)}^2
		\ge
		2\mu \alpha_{\Omega,\Gamma_D}^2 |\mathbf{u}|_{H^1(\Omega)}^2.
\end{equation*}
Using a standard Poincaré--Friedrichs inequality, cf. \cite[Proposition~5.3.4]{BrennerScott}, we know that
\begin{equation*}
		C_F\|\mathbf{u}\|_{H^1(\Omega)}
		\le
		|\mathbf{u}|_{H^1(\Omega)}
		\le
		\|\mathbf{u}\|_{H^1(\Omega)}
		\quad \mbox{for all}\quad \mathbf u \in \mathbf V,
\end{equation*}
where the constant $C_F>0$ depends on $\Omega$ and $\Gamma_D$. Since the Poincaré--Friedrichs inequality guarantees that $|\cdot|_{H^1(\Omega)}$ is actually a norm on $H^1_{0,D}(\Omega)$, we state the following estimates using the \emph{norm} $|\cdot|_{H^1(\Omega)}$.

An estimate of the kind of~\eqref{eq:korn} is obtained using Korn's second inequality. The following Lemma~\ref{lem:korngen} is a handy statement that allows to the construction of various kinds of Korn's inequalities, which we will use throughout this paper.

\begin{remark}
Although the constants of both the Korn and the Poincaré--Friedrichs inequality depend on the size of $\Omega$, the dependence of the Poincaré--Friedrichs inequality is less severe since the $L^2$-norm and the $H^1$-seminorm scale differently with the diameter of $\Omega$. This is different for Korn's inequality: $\|\epsilon(\cdot)\|_{L^2(\Omega)}$ and $|\cdot|_{H^1(\Omega)}$ both are invariant under scaling (or scale like $(\mathrm{diam}\, \Omega)^{d/2-1}$ for $d\ne2$).
\end{remark}

The following Lemma is used to show~\eqref{eq:korn}.

\begin{lemma}\label{lem:korngen}
		Let $\ell(\mathbf u) : [H^1(\Omega)]^2\to \mathbb R$ be a bounded
		linear operator such that $\ell(\mathbf r) \ne 0$ for
		\begin{equation} \label{eq:rigid}
				\mathbf r(x,y) := \begin{pmatrix}
					-y \\ x
				\end{pmatrix}.
		\end{equation}
		Then, there is a constant $\alpha_{\Omega,\ell}>0$ that only depends on
		$\Omega$ and $\ell$ such that
		\begin{equation}
				\|\epsilon(\mathbf v) \|_{L^2(\Omega)} + |\ell(\mathbf v)|
				\ge \alpha_{\Omega,\ell} |\mathbf v|_{H^1(\Omega)}
				\quad \mbox{for all}\quad
				\mathbf v \in [H^1(\Omega)]^2.
		\end{equation}
\end{lemma}
A poof of this Lemma is given at the end of this section.
Since the proof is not constructive, we cannot give a lower bound for
$\alpha_{\Omega, \ell}$ in general. Note that the conditions of the Lemma are satisfied for $\ell_{\Gamma_D}(\mathbf v):= (\mathbf v, \mathbf r)_{L^2(\Gamma_D)}$ with $\mathbf r$ as in~\eqref{eq:rigid}; thus we obtain~\eqref{eq:korn} with a constant $\alpha_{\Omega,\Gamma_D}>0$ that depends on $\Omega$ and $\Gamma_D$.

\begin{remark}
	A simple scaling argument shows that the constant $\alpha_{\Omega,\Gamma_D}$ does not
	depend on the diameter of $\Omega$, but only on its shape.
\end{remark}

Since the assumptions of the Lax-Milgram lamma hold, we know existence and uniqueness of a solution $\mathbf u$ to Problem~\ref{prob1} and $\mathbf u_h$ to Problem~\ref{prob2}. Using Ceá's lemma (in its version \cite[Remark~2.8.5]{BrennerScott} for symmetric bilinear forms), we can bound the discretization error $| \mathbf u - \mathbf u_h |_{H^1(\Omega)}$ by the approximation error as follows:
\[
		| \mathbf u - \mathbf u_h |_{H^1(\Omega)}
		\le \alpha_{\Omega,\Gamma_D}^{-1}\, \sqrt{\frac{\mu + \lambda}{\mu}}
		\inf_{\mathbf v_h \in \mathbf V}
		| \mathbf u - \mathbf v_h |_{H^1(\Omega)}.
\]
This can be further estimated using an approximation error estimate.
\begin{lemma}\label{lem:approx}
Let $q\in \{1,\cdots,\degree\}$ and assume that the geometry functions $\mathbf G_k$ are smooth enough such that there is a constant $\widetilde C_2$ such that
\begin{equation}\label{eq:approxcond}
|\mathbf v \circ \mathbf G_k|_{H^{q+1}(\widehat\Omega)} \le \widetilde C_2 H_k^q
\|\mathbf v\|_{H^{q+1}(\Omega^{(k)})}
\quad\mbox{ for all }\quad \mathbf v\in [H^{q+1}(\Omega)]^2.
\end{equation}
Then, the estimate
\begin{equation}\label{eq:discrerror}
	\inf_{\mathbf v_h \in \mathbf V}| \mathbf v - \mathbf v_h |_{H^1(\Omega)}^2
	\lesssim \widetilde C_2^2 \pi^{-2q} \sum_{k=1}^K h_k^{2q} \| \mathbf v \|_{H^{q+1}(\Omega^{(k)})}^2
	\;\mbox{holds for all}\;
	\mathbf v\in [H^{q+1}(\Omega)]^2.
\end{equation}
\end{lemma}
This lemma is a slight extension of~\cite[Theorem~3.6]{takacs2018robust}; for completeness, a proof is given at the end of this section. Using this lemma, we immediately obtain the discretization error estimate
\begin{equation}\label{eq:discrerror}
	| \mathbf u - \mathbf u_h |_{H^1(\Omega)}^2
	\lesssim \widetilde C_2^2 \alpha_{\Omega,\Gamma_D}^{-2}\,  \frac{\mu + \lambda}{\mu}
	\pi^{-2q} \sum_{k=1}^K h_k^{2q} \| \mathbf u \|_{H^{q+1}(\Omega^{(k)})}^2.
\end{equation}
We observe that the convergece deteriorates if $\alpha_{\Omega,\Gamma_D} \to 0$ or $\mu/\lambda\to0$. The latter happens in case of almost incompressible materials; incompressible materials correspond to the limit case. This effect is known as material locking. To avoid material locking, one usually considers a saddle point formulation of the elasticity equations, however, this is not in the scope of this paper.

The deterioration of the convergence in case $\alpha_{\Omega,\Gamma_D}$ is too small, is usually known as geometry locking. In the following, we show cases where this can happen and show that the actual convergence might still be good, even if $\alpha_{\Omega,\Gamma_D} \to 0$. These discussions are done for a spacial case. Consider a rectangular domain
$\Omega_K :=(0,K) \times (0,1)$ with $K\in\{1,2,3,\cdots\}$, consisting of square patches $\Omega^{(k)} :=(k-1,k)\times(0,1)$ for $k=1,\ldots,K$, i.e.,
\[
		\overline {\Omega_K} = \overline{(0,K) \times (0,1)}
		= \bigcup_{k=1}^K \overline{\Omega^{(k)}}
		= \bigcup_{k=1}^K \overline{(k-1,k)\times(0,1)}
		,
\]
see Figure~\ref{fig:rect}. The patches are parameterized in the canonical way by $\mathbf{G}_k(\widehat x,\widehat y) := (\widehat x+k-1,\widehat y)^\top$. We consider the linearized elasticity equation with homogeneous Dirichlet boundary conditions on $\Gamma_D:=\Gamma^{(0)}$ and Neumann boundary conditions on $\Gamma_N:=\partial\Omega\setminus \Gamma_D$,
the remainder of the boundary.

\begin{figure}[h]
\begin{center}
	\scalebox{.7}{\begin{tikzpicture}
		 \draw[black!50,->,dashed] (-.2,-1) -- (11, -1) node[below] {\large $\color{black} x$};
		 \draw[black!50,->,dashed] (0,-1.2) -- (0, 1.6) node[left] {\large $\color{black}  y$};
		 \draw[line width=.3ex] (-.02,1) -- (10.02,1);
		 \draw[line width=.3ex] (-.02,-1) -- (10.02,-1);
		 \draw[line width=.3ex] (0,-1) -- (0,1);
		 \draw[line width=.3ex] (2,-1) -- (2,1);
		 \draw [line width=.3ex](4,-1) -- (4,1);
		 \draw[line width=.3ex] (6,-1) -- (6,1);
		 \draw[line width=.3ex] (8,-1) -- (8,1);
		 \draw[line width=.3ex] (10,-1) -- (10,1);
		 \draw (1,0) node {\large  $\Omega^{(1)}$};
		 \draw (3,0) node {\large $\Omega^{(2)}$};
		 \draw (5,0) node {\large $\Omega^{(3)}$};
		 \draw (7,0) node {\large $\Omega^{(4)}$};
		 \draw (9,0) node {\large $\Omega^{(5)}$};

		 \draw (-.2,-.5) node {\large $\Gamma_D=\Gamma^{(0)}$};
		 \draw (2.35,-.5) node {\large $\Gamma^{(1)}$};
		 \draw (4.35,-.5) node {\large $\Gamma^{(2)}$};
		 \draw (6.35,-.5) node {\large $\Gamma^{(3)}$};
		 \draw (8.35,-.5) node {\large $\Gamma^{(4)}$};
		 \draw (10.35,-.5) node {\large $\Gamma^{(5)}$};

	\end{tikzpicture}}
\end{center}
\caption{Rectangular domain, consisting of $K=5$ square patches}
\label{fig:rect}
\end{figure}
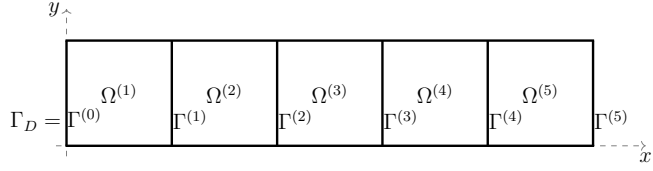

The following Lemma states that the constant
\begin{equation}
		\nonumber
		\alpha_{\Omega_K,\Gamma_D} := \inf_{\mathbf u\in H^1_{0,D}(\Omega_K)}
		\frac{\|\epsilon(\mathbf u)\|_{L^2(\Omega_K)}}
		{|\mathbf u|_{H^1(\Omega_K)}}
\end{equation}
for the corresponding Korn inequality grows linearly with the number of patches.
\begin{lemma}
		For all $K\in \{1,2,3,\ldots\}$, we have $\alpha_{\Omega_K,\Gamma_D} \eqsim K^{-1}$.
\end{lemma}
\begin{proof}
		An upper bound for $\alpha_{\Omega_K,\Gamma_D}$ is obtained for the choice
		$\mathbf v(x,y) := ((2y-1)x, -x^2)^\top$, which obviously satisfies the Dirichlet
		conditions. For this choice, we can explicitly
		compute $\alpha_{\Omega_K,\Gamma_D}^{-1} = \sqrt{1 + 2 K^2}\ge K$. (Note that for $\degree \ge 2$, the function $\mathbf v(x,y)$ is in the spline space $\mathbf V$, which means that the constant $\alpha_{\Omega_K,\Gamma_D}$ cannot be improved by restricting to the discrete space $\mathbf V$.)

		For a lower bound, we use Korn's second inequality for the
		parameter domain $\widehat\Omega:=(0,1)^2$.
		Let $\mathbf v(x,y)=(v_1(x,y),v_2(x,y))^\top$ be arbitrary and define
		$\mathbf w(x,y)=(w_1(x,y),w_2(x,y))^\top:=(v_1(Kx,y),\allowbreak K^{-1}v_2(Kx,y))^\top$. Then, we have using
		simple scaling arguments
		\begin{align*}
				&\frac{|\mathbf v|_{H^1(\Omega)}^2}{\|\epsilon(\mathbf v)\|_{L^2(\Omega)}^2}
			 \le 1+
			 \frac{
			 \|\partial_y v_1\|_{L^2(\Omega)}^2
			 +\|\partial_x v_2\|_{L^2(\Omega)}^2
			 }{
			  \|\partial_x v_1\|_{L^2(\Omega)}^2
			 +\tfrac12 \|\partial_y v_1 + \partial_x v_2\|_{L^2(\Omega)}^2
			 +\|\partial_y v_2\|_{L^2(\Omega)}^2
			 }
			 \\&\quad \le 1+
			 \frac{
			 \|\partial_y w_1\|_{L^2(\widehat\Omega)}^2
			 +\|\partial_x w_2\|_{L^2(\widehat\Omega)}^2
			 }{
			  K^{-2}\|\partial_x w_1\|_{L^2(\widehat\Omega)}^2
			 +\tfrac12 \|\partial_y w_1 + \partial_x w_2\|_{L^2(\widehat\Omega)}^2
			 +K^2\|\partial_y w_2\|_{L^2(\widehat\Omega)}^2
			 }\\&\quad
			 \lesssim K^2 \frac{|\mathbf w|_{H^1(\widehat\Omega)}^2}{\|\epsilon(\mathbf w)\|_{L^2(\widehat\Omega)}^2}
			 \le K^2 \alpha_{\widehat\Omega}^2,
		\end{align*}
		where $\alpha_{\widehat\Omega}>0$ is the constant resulting from
		the use of Korn's inequality on the parameter domain $\widehat \Omega$.
		Since this constant is independent of the actual
		geometry (particularly the number of patches $K$), we have $\alpha_{\widehat\Omega}\eqsim 1$. This finishes the proof.
\end{proof}

Based on these statements, the discretization error estimate~\eqref{eq:discrerror} degrades for $K\to\infty$.
For $\lambda=0$, the following lemma shows that the proposed isogeometric
discretization does not suffer from geometry locking. This estimate is restricted to
$\degree \ge 2$, which hints that it make sense to consider discretizations other than
the standard linear finite elements.
\begin{lemma}
	Let $\lambda=0$ and $\degree \ge 2$.
	For all $K\in \{1,2,3,\ldots\}$, we have
	\begin{equation}\label{eq:kornfreeresult}
			| \mathbf u - \mathbf u_h |_{H^1(\Omega)}^2
			\lesssim  \inf_{\mathbf v_h \in \mathbf V} |\mathbf u- \mathbf v_h|_{H^1(\Omega)}^2,
	\end{equation}
	where $\mathbf u$ and $\mathbf u_h$ are the solutions to the Problem~\ref{prob1} and the discretized Problem~\ref{prob2}, respectively. Under the assumptions of Lemma~\ref{lem:approx}, we have further
	\[
		| \mathbf u - \mathbf u_h |_{H^1(\Omega)}^2
		\lesssim \widetilde C_2^2 \pi^{-2q} \sum_{k=1}^K h_k^{2q} \| \mathbf u \|_{H^{q+1}(\Omega^{(k)})}^2.
	\]
\end{lemma}
\begin{proof}
	Let $\ell(w_1,w_2):=\int_{0}^{1} (2y-1)\, w_1(1,y) \,\mathrm dy$. Due to the trace inequality, it is well-defined and bounded in $[H^1(\widehat\Omega)]^2$. Moreover, we immediately observe that $\ell(\mathbf r)\ne0$ for $\mathbf r$ as in~\eqref{eq:rigid}. Thus, Lemma~\ref{lem:korngen} yields
	\[
			| \mathbf w |_{H^1(\widehat\Omega)}
			\le \alpha_{\widehat\Omega,\ell}^{-1}  \left(
			\| \epsilon( \mathbf w ) \|_{L^2(\widehat\Omega)}
			+
			|\ell(\mathbf w)|
			\right),
	\]
	where $\alpha_{\widehat\Omega,\ell}>0$ is the constant from Lemma~\ref{lem:korngen}. Since this constant is independent of the actual
	geometry (particularly the number of patches $K$), we have $\alpha_{\widehat\Omega,\ell}\eqsim 1$.
	Since this statement is invariant under translations, we obtain also
	\begin{equation}\label{eq:robustproof}
			| \mathbf w |_{H^1(\Omega)}^2
						\lesssim \left(
							\| \epsilon( \mathbf w ) \|_{L^2(\Omega)}^2
							+
							\sum_{k=1}^K
							\left(\int_{0}^{1} (2y-1)\, w_1(k,y) \,\mathrm dy\right)^2
							\right)
	\end{equation}
	for all $\mathbf w(\mathbf x)=(w_1(\mathbf x), w_2(\mathbf x))^\top$.
	Define
	\[
			\psi_k(x,y) :=
			\begin{cases}
					((2y-1)x, -x^2)^\top & \mbox{for } x<k,\\
				     ((2y-1)k, k^2-2kx)^\top & \mbox{for } x\ge k
			\end{cases}
	\]
	and observe that
	\[
			\epsilon(\psi_k) = \begin{pmatrix}
					2y-1 & 0 \\ 0 & 0
			\end{pmatrix}
			\chi_{x<k},
	\]
	where $\chi_{x<k}$ is the characteristic function, i.e., it takes the value $1$ for
	$x<k$ and vanishes otherwise. Since $\psi_k \in \mathbf V$,
	the Galerkin orthogonality implies
	\[
			\int_\Omega \epsilon(\mathbf u-\mathbf u_h) : \epsilon(\psi_k) \,\mathrm d\mathbf x
				= 0
				   \quad \mbox{for all}\quad k\in\{1,\ldots,K\}.
	\]
	Let $\mathbf w:=\mathbf u-\mathbf u_h=(w_1,w_2)^\top$. Using the explicit representation of $\psi_k$ and the Dirichlet conditions on $\Gamma_D=\Gamma^{(0)}$, we obtain
	\[
			0 =  \int_0^k \int_{0}^{1} (2y-1)\partial_x  w_1(x,y) \,\mathrm dy
			\,\mathrm dx
			=  \int_{0}^{1} (2y-1)\, w_1(k,y) \,\mathrm dy.
	\]
	In combination with~\eqref{eq:robustproof}, this yields
	$
		| \mathbf u - \mathbf u_h |_{H^1(\Omega)}^2
						\lesssim \| \epsilon( \mathbf u -\mathbf u_h ) \|_{L^2(\Omega)}^2.
	$
	Using Galerkin orthogonality, we immediately obtain~\eqref{eq:kornfreeresult}.
	The second results directly follows using Lemma~\ref{lem:approx}.
\end{proof}

We are interested in a fast solver for this problem, which does not degrade for domains with small Korn constant. For the solution, we aim for a domain decomposition approach, where the convergence rates only depend on the Korn constants for the individual patches, which serve as substructures for the domain decomposition solver; specifically, we require that
\begin{equation}\label{eq:localkorn0}
	\|\epsilon(\mathbf{v})\|_{L^2(\kbox{\Omega})}
	\ge
	\alpha_k
	|\mathbf{v}|_{H^1(\kbox{\Omega})}
	\; \mbox{for all}\;
	\mathbf{v} \in [H^1({\kbox{\Omega}})]^2
	\; \mbox{with} \;
	\int_{\kbox{\Omega}} \curl \mathbf v \; \mathrm d \mathbf x = 0
	,
\end{equation}
which follows from \cite[eq.~(3.0.1) and (3.1.1)]{AcostaDuran:2017}. Certainly, the
constant $\alpha_k$ only depends on the shape of $\Omega^{(k)}$ and is independent
on the number of patches and the shape of the overall domain $\Omega$.

Finally, we give proofs for Lemmas~\ref{lem:korngen} and \ref{lem:approx}.

\begin{proofof}{of Lemma~\ref{lem:korngen}}
		Let $\mathbf v \in [H^1(\Omega)]^2$ arbitrary but fixed.
		\cite[eq.~(3.0.1) and (3.1.1)]{AcostaDuran:2017} states that there is a constant $\alpha_\Omega>0$ such that
		\begin{equation}\label{eq:globalkorn0}
			\|\epsilon(\mathbf{w})\|_{L^2(\Omega)}
			\ge
			\alpha_\Omega
			|\mathbf{w}|_{H^1(\Omega)}
			\; \mbox{for all}\;
			\mathbf{w} \in [H^1({\Omega})]^2
			\; \mbox{with} \;
			\int_{\Omega} \curl \mathbf w \; \mathrm d \mathbf x = 0
			.
		\end{equation}
		For $\mathbf r$ as in~\eqref{eq:rigid}, we have
		$\curl \mathbf r=2$.  For $\rho:=\tfrac 12 \int_{\Omega} \curl \mathbf v \; \mathrm d \mathbf x$, we have $\mathbf v = \rho \, \mathbf r + \mathbf w$, where
		$\int_{\Omega} \curl \mathbf w \; \mathrm d \mathbf x = 0$.
		Using the triangle inequality, the linearity of $\ell$ and its boundedness, we obtain
		\begin{align*}
			|\mathbf v|_{H^1(\Omega)}
			& \le
			|\mathbf w|_{H^1(\Omega)}
			+ |\rho\,\mathbf r|_{H^1(\Omega)}
			=
			|\mathbf w|_{H^1(\Omega)}
		+  \frac{|\mathbf r|_{H^1(\Omega)}}{|\ell(\mathbf r)|}	|\ell( \rho \mathbf r)|
		 	\\&
			\le
			|\mathbf w|_{H^1(\Omega)}
					+  \frac{|\mathbf r|_{H^1(\Omega)}}{|\ell(\mathbf r)|}	|\ell(\textbf w)|
					+  \frac{|\mathbf r|_{H^1(\Omega)}}{|\ell(\mathbf r)|}	|\ell( \textbf v)|
			\\&
			\le
			\left( 1 +  \frac{|\mathbf r|_{H^1(\Omega)}}{|\ell(\mathbf r)|} \|\ell\|_{[H^1(\Omega)]^*} \right)
				|\mathbf w|_{H^1(\Omega)}
				+\frac{|\mathbf r|_{H^1(\Omega)}}{|\ell(\mathbf r)|}  |\ell( \mathbf v)| ,
		\end{align*}
		where $\|\ell\|_{[H^1(\Omega)]^*}$ denotes the dual norm of $\ell$.
		Using~\eqref{eq:globalkorn0}, $\epsilon(\mathbf v-\mathbf w)
		=\epsilon(\rho\mathbf r)=0$, and $\alpha_\Omega \le 1$, we further have
		\begin{align*}
				|\mathbf v|_{H^1(\Omega)}
				& \le
				\alpha_\Omega^{-1}
				\left( 1 +  \frac{|\mathbf r|_{H^1(\Omega)}}{|\ell(\mathbf r)|}
				\max\{1,\|\ell\|_{[H^1(\Omega)]^*} \} \right)
					\left( \|\epsilon(\mathbf v)\|_{L^2(\Omega)} +  |\ell( \mathbf v)|\right),
		\end{align*}
		which finishes the proof for
		$\alpha_{\Omega,\ell}^{-1}= \alpha_\Omega
			\left( 1 +  \frac{|\mathbf r|_{H^1(\Omega)}}{|\ell(\mathbf r)|}
			\max\{1,\|\ell\|_{[H^1(\Omega)]^*} \} \right)^{-1}$.
\end{proofof}

\begin{proofof}{of Lemma~\ref{lem:approx}}
	This lemma is a slight generalization of \cite[Theorem~3.6]{takacs2018robust}. As in \cite{takacs2018robust}, we construct the function $\mathbf v_h$ patchwise. We choose $\mathbf v_h \circ \mathbf G_k := \widehat{\Pi}_k ( \mathbf u \circ \mathbf G_k)$ as in \cite[Eq.~(3.11)]{takacs2018robust}, where $\widehat{\Pi}_k$ is the projector from~\cite[Eq.~(3.9)]{takacs2018robust}. This method indeed gives a function $\mathbf v_h\in \mathbf V$, see~\cite[Lemma~3.4]{takacs2018robust}. For deriving the error estimates themselves, we do not use the $\degree$-robust estimates from~\cite{TakacsTakacs:2016}, but the more recent $\degree$-robust estimates from~\cite[Theorem 3.1]{SandeManniSpeleers:2019}. Together with~\eqref{eq:ass:nabla},
	and \eqref{eq:gridsize}, we obtain
	\begin{align*}
		\inf_{\mathbf v_h \in \mathbf V}| \mathbf u - \mathbf v_h |_{H^1(\Omega)}^2
		& \lesssim  \sum_{k=1}^K
		|  (I - \widehat{\Pi}_k) (\mathbf u\circ \mathbf G_k) |_{H^1}^2
		\le  \pi^{-2q} \sum_{k=1}^K \widehat h_k^{2q}
		|  \mathbf u\circ \mathbf G_k |_{H^{q+1}(\widehat \Omega)}^2,
	\end{align*}
	form which the desired result follows using~\eqref{eq:approxcond}.
\end{proofof}

\section{A IETI-DP solver for the elasticity problem}
\label{sec:IETISolver}

In this section, we outline the setup of the proposed IETI-DP method for solving the discretized elasticity problem~\eqref{eq:discelast}.
For the IETI-DP method, we have to assemble the variational problem
locally, i.e., we have
\begin{align} \nonumber
	\kbox{a} (\kbox{\mathbf{u}}, \kbox{\mathbf{v}})
	&:= 2\mu (\epsilon(\kbox{\mathbf{u}}), \epsilon(\kbox{\mathbf{v}}))_{L^2(\Omega^{(k)})}
	+ \lambda (\divergence \kbox{\mathbf{u}}, \divergence \kbox{\mathbf{v}})_{L^2(\Omega^{(k)})} \\
	\langle \kbox{F}, \kbox{\mathbf{v}} \rangle
	&:=(\mathbf{f},\kbox{\mathbf{v}})_{L^2(\Omega^{(k)})} + (\mathbf{g}_N, \kbox{\mathbf{v}})_{L^2(\kbox{\Gamma_N})},
\end{align}
where $\kbox{\Gamma_N}:=\Gamma_N\cap\partial\kbox{\Omega}$.
If the problems would not be coupled, we would have
$\kbox{\mathbf{u}} \in \kbox{\mathbf V}$ with
\[
		\kbox{a} (\kbox{\mathbf{u}}, \kbox{\mathbf{v}})
		\stackrel{!}{=}
		\langle \kbox{F}, \kbox{\mathbf{v}} \rangle
		\quad \mbox{for all}\quad \kbox{\mathbf{v}} \in \kbox{\mathbf V}.
\]
By fixing bases, we obtain linear systems
\begin{equation}
  \label{eq:stiffness1}
    \kbox{A} \kbox{\underline{\mathbf{u}}}
    \stackrel{!}{=} \underline{\mathbf{f}}^{(k)}.
\end{equation}
Here and in what follows, underlined quantities refer to the coefficient
representations of the corresponding functions.
We represent the spaces $\mathbf{V}^{(k)}$ as direct sum
$\mathbf{V}^{(k)} = \mathbf{V}^{(k)}_{\Gamma} \oplus \mathbf{V}^{(k)}_{\I}$, where $\mathbf{V}^{(k)}_{\I}$ is spanned by the basis functions which vanish on the interfaces and $\mathbf{V}^{(k)}_{\Gamma}$ is spanned by the remaining functions, i.e., the functions that are active on the interfaces (which includes the primal degrees of freedom). Assuming a corresponding ordering of the basis functions, we have
\begin{equation}
  \label{eq:patchDecomposedMatrix}
		\begin{pmatrix}
		  \kbox{A_{\Gamma\Gamma}} & \kbox{A_{\Gamma\I}} \\
		  \kbox{A_{\I\Gamma}}     & \kbox{A_{\I\I}}     \\
		\end{pmatrix}
        \begin{pmatrix}
          \underline{\mathbf{u}}_\Gamma^{(k)}\\
          \underline{\mathbf{u}}_\I^{(k)}
        \end{pmatrix}
        \stackrel{!}{=}
        \begin{pmatrix}
          \underline{\mathbf{f}}_\Gamma^{(k)} \\
          \underline{\mathbf{f}}_\I^{(k)}
        \end{pmatrix}.
\end{equation}
These local systems correspond to pure Neumann problems, unless the corresponding patch contributes to the Dirichlet boundary. Therefore, these systems are not uniquely solvable since the constant is in the null space of $\kbox{A}$.
To ensure that the system matrices of the patch-local problems are non-singular, we introduce the function values on each of the corners of the patch are primal degrees of freedom, i.e.,
\[
			\mathbf u^{(k)}(\textbf x_j) = \mathbf u^{(\ell)}(\textbf x_j)
			\qquad\mbox{for all}\qquad
			j=1,\ldots,J
			\mbox{ and }
			k,\ell \in \mathcal N_{\textbf x}(j).
\]
The convergence analysis that we develop carries over automatically to the case where additional primal degrees of freedom are introduced, like the edge averages of $\mathbf u$ or its normal component;  we consider such choices in the numerical experiments presented in Section~\ref{sec:num}. For the theoretical considerations, we require only the corner values to be chosen as primal
degrees of freedom.

The corresponding Lagrangian multipliers are denoted by $\kbox{\mu}$. We obtain
\begin{equation}
  \label{eq:localElastAverage}
    \bar{A}^{(k)}\bar{\underline{\mathbf{u}}}^{(k)} :=
		\begin{pmatrix}
		  \kbox{A_{\Gamma\Gamma}}  &\kbox{A_{\Gamma\I}} &\kbox{C}^\top\\
		  \kbox{A_{\I\Gamma}}      &\kbox{A_{\I\I}}     &0\\
		  \kbox{C}                &0                 &0\\
		\end{pmatrix}
        \begin{pmatrix}
          \kbox{\underline{\mathbf{u}}_\Gamma}\\
          \kbox{\underline{\mathbf{u}}_\I}\\
          \kbox{\underline{\mu}}
        \end{pmatrix}
       \stackrel{!}{=}
        \begin{pmatrix}
          \kbox{\underline{\mathbf{f}}_\Gamma} \\
          \kbox{\underline{\mathbf{f}}_\I} \\
          0
        \end{pmatrix}=:\bar{\underline{\mathbf{f}}}^{(k)},
\end{equation}
where $\kbox{C}$ evaluates the primal degrees of freedom associated to the patch $\kbox\Omega$,
such that $\kbox{C} \kbox{\underline{\mathbf{u}}_\Gamma} = 0$.
The continuity between the patches is enforced
by the matrices $\kbox{B}$. The term
$
		\sum_{k=1}^K \kbox{B} \underline{\mathbf u}_\Gamma^{(k)}
$
evaluates to a vector containing the differences of the coefficients of any two matching
basis functions. Here, we do not include
the vertex values of (see Figure~\ref{fig:nonredundant}) since these are primal
degrees of freedom. The relation
\[
  \sum_{k=1}^K \kbox{B} \underline{\mathbf u}_\Gamma^{(k)} = 0
\]
guarantees the continuity across the patches.
\begin{figure}
	\begin{center}
	  \begin{tikzpicture}
			\fill[gray!15] (-0.2,0) -- (1.5,0) -- (1.5,1.7) -- (-0.2,1.7);
			\fill[gray!15] (-0.2,-1) -- (1.5,-1) -- (1.5,-2.7) -- (-0.2,-2.7);
			\fill[gray!15] (4.2,0) -- (2.5,0) -- (2.5,1.7) -- (4.2,1.7);
			\fill[gray!15] (4.2,-1) -- (2.5,-1) -- (2.5,-2.7) -- (4.2,-2.7);

			\draw (-0.2,0) -- (1.5,0) -- (1.5,1.7) node at (0.7,0.8) {$\Omega^{(1)}$};
			\draw (-0.2,-1) -- (1.5,-1) -- (1.5,-2.7) node at (0.7,-1.85) {$\Omega^{(3)}$};
			\draw (4.2,0) -- (2.5,0) -- (2.5,1.7) node at (3.4,0.8) {$\Omega^{(2)}$};
			\draw (4.2,-1) -- (2.5,-1) -- (2.5,-2.7) node at (3.4,-1.85) {$\Omega^{(4)}$};

			\draw (1.5,0) node[circle, fill, inner sep = 2pt] (A1) {};
			\draw (1.5,0.75) node[circle, fill, inner sep = 2pt] (A2) {};
			\draw (1.5,1.5) node[circle, fill, inner sep = 2pt] (A3) {};
			\draw (0.75,0) node[circle, fill, inner sep = 2pt] (A4) {};
			\draw (0,0) node[circle, fill, inner sep = 2pt] (A5) {};

			\draw (2.5,0) node[circle, fill, inner sep = 2pt] (B1) {};
			\draw (2.5,0.75) node[circle, fill, inner sep = 2pt] (B2) {};
			\draw (2.5,1.5) node[circle, fill, inner sep = 2pt] (B3) {};
			\draw (3.25,0) node[circle, fill, inner sep = 2pt] (B4) {};
			\draw (4,0) node[circle, fill, inner sep = 2pt] (B5) {};

			\draw (1.5,-1) node[circle, fill, inner sep = 2pt] (C1) {};
			\draw (1.5,-1.75) node[circle, fill, inner sep = 2pt] (C2) {};
			\draw (1.5,-2.5) node[circle, fill, inner sep = 2pt] (C3) {};
			\draw (0.75,-1) node[circle, fill, inner sep = 2pt] (C4) {};
			\draw (0,-1) node[circle, fill, inner sep = 2pt] (C5) {};

			\draw (2.5,-1) node[circle, fill, inner sep = 2pt] (D1) {};
			\draw (2.5,-1.75) node[circle, fill, inner sep = 2pt] (D2) {};
			\draw (2.5,-2.5) node[circle, fill, inner sep = 2pt] (D3) {};
			\draw (3.25,-1) node[circle, fill, inner sep = 2pt] (D4) {};
			\draw (4,-1) node[circle, fill, inner sep = 2pt] (D5) {};

			\draw[<->, line width = 1pt, latex-latex]
			(A2) edge (B2) (A3) edge (B3)
			(A4) edge (C4) (A5) edge (C5)
			(B4) edge (D4) (B5) edge (D5)
			(C2) edge (D2) (C3) edge (D3);
			\end{tikzpicture}
	  \caption{Enforcing continuity across the patches. The corners are excluded.}
      \label{fig:nonredundant}
	\end{center}
\end{figure}
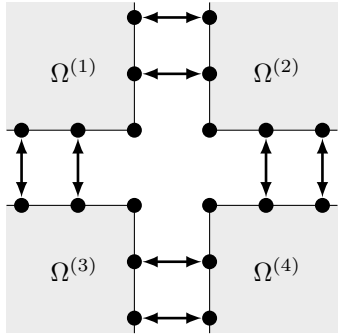

Moreover, we introduce the primal problem, i.e., the global problem for the primal degrees of freedom. We use a $A^{(k)}$-orthogonal basis for the primal degrees of freedom. This basis is represented in terms of the basis functions of the basis for $\mathbf V_\Gamma^{(k)}\times \mathbf V_\I^{(k)}$ using the matrix $\Psi^{(k)}$, which are the solution of the system
\begin{equation}\label{eq:basisdef}
		\begin{pmatrix}
		  \kbox{A_{\Gamma\Gamma}}  &\kbox{A_{\Gamma\I}} &\kbox{C}^\top\\
		  \kbox{A_{\I\Gamma}}      &\kbox{A_{\I\I}}     &0\\
		  \kbox{C}                &0                 &0\\
		\end{pmatrix}
		\begin{pmatrix}
					\kbox{\Psi_{\Gamma}} \\
					\kbox{\Psi_{\I}} \\
					\kbox{\Phi} \\
		\end{pmatrix}
		=
		\begin{pmatrix}
					0 \\
					0 \\
					\kbox{R}
		\end{pmatrix},
\end{equation}
where $\kbox{R}$ is a boolean matrix that selects the primal degrees of freedom which are active on the patch~$\Omega^{(k)}$.
We have $\kbox{R}\in \mathbb{R}^{N^{(k)}_{\Pi}\times N_\Pi}$, where $N^{(k)}_{\Pi}$ is the number of primal degrees of freedom corresponding to the $k$'th patch (thus $N^{(k)}_{\Pi}=8$ if the patch does not contribute to the Dirichlet boundary) and $N_\Pi$ is the overall number of primal degrees of freedom.

We define the system matrix, right-hand side and jump matrix for the primal problem as
\[
A_\Pi := \sum^K_{k=1} \kbox{\Psi}^\top \kbox{A}\kbox{\Psi},\quad
\underline{\mathbf f}_\Pi := \sum^K_{k=1} \kbox{\Psi}^\top \kbox{\underline{\bar{\mathbf f}}}
\quad \text{and}\quad
B_\Pi := \sum^K_{k=1} \kbox{B}\kbox{\Psi}.
\]
Finally we are able to write down the overall IETI-DP system:
\begin{equation}
  \label{eq:saddleForm}
		\begin{pmatrix}
				\kkbox{\bar{A}}{1}   &&&& \kkbox{\bar{B}}{1}^\top \\
				& \ddots              &&& \vdots                  \\
				&& \kkbox{\bar{A}}{K}  && \kkbox{\bar{B}}{K}^\top \\
				&&& {A}_\Pi         & {B}_\Pi^\top        \\
		    \kkbox{\bar{B}}{1} & \hdots & \kkbox{\bar{B}}{K} & {B}_\Pi  \\
		\end{pmatrix}
		\begin{pmatrix}
			\kkbox{\underline{\bar{\mathbf x}}}1\\
			\vdots\\
			\kkbox{\underline{\bar{\mathbf x}}}K\\
			\underline{{\mathbf x}}_\Pi \\
			\underline{\lambda}
		\end{pmatrix}
		=
		\begin{pmatrix}
			\kkbox{\underline{\bar{\mathbf f}}}1\\
			\vdots\\
			\kkbox{\underline{\bar{\mathbf f}}}K\\
			\underline{\mathbf f}_\Pi \\
			0
		\end{pmatrix}.
\end{equation}
For solving this linear system, we take the Schur complement with
respect to the Lagrange multipliers $\underline{\lambda}$.
This means that we solve
\begin{equation}
  \label{eq:lambdaSys}
  F\underline{\lambda} = \underline{g},
\end{equation}
where
\begin{equation*}
	F := {B}_\Pi {A}^{-1}_\Pi {B}_\Pi^\top+\sum^{K}_{k = 1}  \kbox{\bar B} \kbox{\bar{A}}^{-1}  \kbox{\bar B}^{\top}
	\;\mbox{and}\;
\underline{g}:= {B}_\Pi {A}^{-1}_\Pi{\underline{\mathbf{f}}}_\Pi + \sum^{K}_{k = 1}  \bar{B}^{(k)} \kbox{\bar{A}}^{-1}\bar{\underline{\mathbf{f}}}^{(k)}
\end{equation*}
The patch-local solutions are then recovered by
\[
\underline{{\mathbf{x}}}_\Pi = {A}_\Pi^{-1}\left(\underline{{\mathbf{f}}}_\Pi- B_\Pi^{\top}\underline{\lambda}\right)
\quad\mbox{and}\quad
\kbox{\underline{\bar{\mathbf{x}}}} = \kbox{\bar{A}}^{-1}\left(\underline{\bar{\mathbf{f}}}^{(k)}- \kbox{\bar B}^{\top}\underline{\lambda}\right).
\]
The final solution is then obtained by distributing $\underline{{\mathbf{x}}}_\Pi$ to the patches.

We solve the system~\eqref{eq:lambdaSys} with a conjugate gradient solver, preconditioned with the scaled Dirichlet preconditioner
\[
		M_{\mathrm{sD}} := \sum_{k=1}^K \kbox{B} \scaling_k^{-1} \kbox{S_{A}} \scaling_k^{-1} \kbox{B}^\top
		,
\]
where $\scaling_k:=2I$ is set up based on the principle of multiplicity scaling and $\kbox{S_A} :=\kbox{A_{\Gamma\Gamma}}-\kbox{A_{\Gamma\I}}\kbox{A_{\I\I}}^{-1}\kbox{A_{\I\Gamma}}$ is the local Schur complement based on the splitting~\eqref{eq:patchDecomposedMatrix}.

\begin{remark}
  In this paper we focus on solving \eqref{eq:lambdaSys} for the symmetric positive definite matrix $F$.
  This requires the exact representation of the inverses of $\kbox{\bar A}$. If these are unavailable, then efficient approximations of the inverses can be used by instead solving \eqref{eq:saddleForm}. For more information see, e.g., \cite{schneckenleitner2021inexact, montardini2022ieti}.
\end{remark}

\section{Condition number analysis for the IETI solver}\label{sec:ieti:theory}

The convergence rates of the conjugate gradient solver can be estimated based on the condition number of the preconditioned system $ M_{\mathrm{sD}} \bar F$, based on the framework introduced in \cite{MandelDohrmannTezaur:2005a} and based on the results from~\cite{SchneckenleitnerTakacs:2020}, where a IETI-DP solver for the Poisson problem has been analyzed in detail.

In this paper, we are interested in a IETI-DP solver for the linearized elasticity problem. The extension of the results from~\cite{SchneckenleitnerTakacs:2020} to the vector-valued is not an issue since all proofs could be reused verbatim. The main difference is that we have to show the coercivity of the bilinear form associated to the linearized elasticity problem, which again requires the use of the Korn inequality. In general, not every patch touches the Dirichlet boundary $\Gamma_D$. Thus, we need a framework that allows the application of the Korn inequality.

Let $\mathbf V_d:=\mathbf V^{(1)} \times \cdots \times \mathbf V^{(K)}$.  The space $\widetilde{\mathbf V} \subseteq \mathbf V_d$ models all functions whose corner values are continuous:
\[
		\widetilde{\mathbf V} :=
		\left\{
				\mathbf v \in \mathbf V_d
				:
				\mathbf v^{(k)}(\textbf x_j)=\mathbf v^{(\ell)}(\textbf x_j)
				\mbox{ for all }
				k,\ell\in\mathcal N_{\textbf x}(j)
				\mbox{ with }
				j\in \{1,\ldots,J\}
		\right\},
\]
where $\mathbf v=(\mathbf v^{(1)},\cdots,\mathbf v^{(K)})$.

In the following, we show that the $H^1$-seminorm can be bounded from above by the semi-norm $\|\epsilon(\cdot)\|_{L^2(\kbox{\Omega})}$ if we guarantee that the corner values vanish. This, of course, is only possible for spaces, where the notation of the corner values is well-defined, like the discrete space $\kbox{\mathbf V}$. Let $\mathcal Q_k:=\{ \mathbf v : \mathbf v \circ \mathbf G_k \mbox{ is bilinear} \}$ be the space of all isogeometric functions with bilinear pre-image. Since a bilinear function is uniquely given by its values at the corners of the unit square $\widehat \Omega$, the following projector is well-defined:
\[
		\kbox{\Pi_F}: \kbox{\mathbf V} \to \mathcal Q_k,
		\quad
		\mbox{such that}
		\quad
		(\kbox{\Pi_F} \kbox{\mathbf v})(\textbf x_\ell) = \kbox{\mathbf v}(\textbf x_\ell)
		\quad \mbox{for all} \quad
		\ell\in \mathcal N_{\textbf x}^{-1}(k).
\]

The following lemma gives upper bounds both for the $H^1$-seminorm and the $L^\infty_0$-seminorm, given by
\[
	| \kbox{\mathbf v} |_{L^\infty_0(\kbox{\Omega})}
	:=
	\inf_{\mu_k\in\mathbb  R^2}
	\| \kbox{\mathbf v} - \mu_k \|_{L^\infty(\kbox{\Omega})},
\]
of $\kbox{\mathbf v}-\kbox{\Pi_F} \kbox{\mathbf v}$, which vanishes on the corners by construction.

\begin{lemma}\label{lem:korn2}
	For all $k=1,\ldots,K$,	the inequalities
	\begin{align*}
			| \kbox{\mathbf v}-\kbox{\Pi_F} \kbox{\mathbf v}|_{H^1(\kbox{\Omega})}
			&\lesssim
			\alpha_k^{-1} \Lambda_k^{1/2}
			\|\epsilon(\kbox{\mathbf v})\|_{L^2(\kbox{\Omega})},
			\\
				| \kbox{\mathbf v}-\kbox{\Pi_F} \kbox{\mathbf v}|_{L_0^\infty(\kbox{\Omega})}
			&\lesssim
				\alpha_k^{-1} \Lambda_k^{1/2}
				\|\epsilon(\kbox{\mathbf v})\|_{L^2(\kbox{\Omega})}
	\end{align*}
	hold for all $\mathbf v\in \widetilde{\mathbf V}$, where
	$\alpha_k$ is the constant from~\eqref{eq:localkorn0},
	$\Lambda_k:=1+\log \degree + \log \frac{H_k}{h_k}$.
\end{lemma}
\begin{proof}
	For any $\mathbf v \in \widetilde{\mathbf V}$, there is a decomposition $\kbox{\mathbf v} = \kbox{\mathbf w} + \rho_k \mathbf r$,
	with $\int_{\kbox{\Omega}} \curl \kbox{\mathbf w}\,\mathrm dx =0$,
	$\mathbf r$ as in~\eqref{eq:rigid} and $\rho_k\in \mathbb R$. Using the triangle
	inequality, we have
	\begin{equation}\label{eq:korn2:triangle}
			|\kbox{\mathbf v} - \kbox{\Pi_F} \kbox{\mathbf v} |_{H^1(\kbox{\Omega})}
			\le
			|\kbox{\mathbf w}  |_{H^1(\kbox{\Omega})}
			+ |\kbox{\Pi_F} \kbox{\mathbf v} - \rho_k \mathbf r |_{H^1(\kbox{\Omega})}.
	\end{equation}
	Since both $\kbox{\Pi_F} \kbox{\mathbf v} $ and $\mathbf r$ are bilinear functions,
	the term $\kbox{\Pi_F} \kbox{\mathbf v} - \rho_k \mathbf r$ is uniquely determined by the
	function values at the corners. Since only the $H^1$-seminorm is considered, only
	deviations from the average matter; thus
	\begin{equation}\label{eq:korn2:est1}
			|\kbox{\Pi_F} \kbox{\mathbf v} - \rho_k \mathbf r |_{H^1(\kbox{\Omega})}
			\lesssim
			\psi_k:=
			\inf_{\mu_k \in \mathbb R^2}
			\sum_{\ell\in \mathcal N_{\mathbf x}^{-1}(k)}
			 |(\kbox{\Pi_F} \kbox{\mathbf v})(\mathbf x_\ell) - \rho_k \mathbf r(\mathbf x_\ell) - \mu_k|.
	\end{equation}
	Since $\kbox{\Pi_F} \kbox{\mathbf v}$ and $\kbox{\mathbf v}$ coincide at the corners, we obtain
	\[
				\psi_k
				=
				\inf_{\mu_k \in \mathbb R^2}
				\sum_{\ell\in \mathcal N_{\mathbf x}^{-1}(k)}
				 |\kbox{\mathbf v}(\mathbf x_\ell) - \rho_k \mathbf r(\mathbf x_\ell) - \mu_k|
				 =
				\inf_{\mu_k \in \mathbb R^2}
				\sum_{\ell\in \mathcal N_{\mathbf x}^{-1}(k)}
			 |\kbox{\mathbf w}(\mathbf x_\ell) - \mu_k|
				 .
		\]
	Using \cite[Lemma~4.14]{SchneckenleitnerTakacs:2020} and the usual relation between
	parameter domain and physical domain and the Poincaré--Friedrichs inequality, we obtain further
	\begin{equation}\label{eq:korn2:est2}
				\psi_k
				\lesssim
				\Lambda_k^{1/2}
				\inf_{\mu_k \in \mathbb R^2}
				\|\kbox{\mathbf w} - \mu_k\|_{H^1(\kbox{\Omega})}\\
				\lesssim
				\Lambda_k^{1/2}
				|\kbox{\mathbf w} |_{H^1(\kbox{\Omega})}
				.
		\end{equation}
		Together with~\eqref{eq:korn2:est1} and \eqref{eq:korn2:triangle}, this yields
		\[
		\begin{aligned}
				|\kbox{\mathbf v} - \kbox{\Pi_F} \kbox{\mathbf v}  |_{H^1(\kbox{\Omega})}
				&\lesssim
				\Lambda_k^{1/2}
				|\kbox{\mathbf w} |_{H^1(\kbox{\Omega})}
				.
		\end{aligned}
		\]
		Finally, \eqref{eq:localkorn0} gives the desired first estimate. For the
		second result, use again the triangle inequality, the fact that the $L_0^\infty$-seminorm of the bilinear function $\kbox{\Pi_F} \kbox{\mathbf v} - \rho_k \mathbf r $ can be bounded by its values at the corner,
		\cite[Lemma~4.14]{SchneckenleitnerTakacs:2020}, and~\eqref{eq:korn2:est1} and~\eqref{eq:korn2:est2} yield
		\[
		\begin{aligned}
				| \kbox{\mathbf v}-\kbox{\Pi_F} \kbox{\mathbf v}|_{L_0^\infty(\kbox{\Omega})}
				\lesssim
				|\kbox{\mathbf w}  |_{L_0^\infty(\kbox{\Omega})}
					+ \psi_k
				\lesssim \Lambda_k^{1/2}
								|\kbox{\mathbf w}  |_{H^1(\kbox{\Omega})},
		\end{aligned}
		\]
		from which the second estimate follows using \eqref{eq:localkorn0}.
\end{proof}

We build our main condition number estimate on the framework from~\cite{MandelDohrmannTezaur:2005a}. To obtain a concise notation, let $\mathbf W_d:=\mathbf W^{(1)} \times \cdots \times \mathbf W^{(K)}$ with $\mathbf W^{(k)} := \{ \mathbf v|_{\partial\Omega^{(k)}} \,:\, \mathbf v\in \mathbf V^{(k)} \}$ be the skeleton space and let
\[
		\widetilde{\mathbf W} :=
		\left\{
				\mathbf w \in \mathbf W_d
				:
				\mathbf w^{(k)}(\textbf x_j)=\mathbf w^{(\ell)}(\textbf x_j)
				\; \forall\;
				k,\ell\in\mathcal N_{\textbf x}(j)
				\mbox{ with }
				j\in \{1,\ldots,J\}
		\right\},
\]
where $\mathbf w=(\mathbf w^{(1)},\cdots,\mathbf w^{(K)})$, be the corresponding space which is continuous at the corners.
Here and in what follows, for any function, say $\kbox{\mathbf w}$
in $\kbox{\mathbf W}$, the corresponding underlined
symbol, here $\kbox{\underline{\mathbf w}}$, denotes the representation
of the corresponding functions with respect to the basis for the space
$\kbox{\mathbf W}$. The overall vector $\underline{\mathbf w}$ is a block-vector containing the local contributions $\kbox{\underline{\mathbf w}}$.
Analogously, $B$ is block-row matrix containing the local matrices $\kbox B$ and  $\scaling$ and $S_A$ are block-diagonal matrices containing the local matrices $\scaling_k$ and $\kbox{S_A}$, respectively.
\cite[Theorem~22]{MandelDohrmannTezaur:2005a} states that the condition number bound depends on the maximum value of the ratio
\begin{equation}\label{eq:ratio}
		\frac{\| \scaling^{-1} B^\top B \underline{\mathbf v} \|_{S_A}^2}
		{\| \underline{\mathbf v} \|_{S_A}^2},
\end{equation}
where the vector $\underline{\mathbf v}$ is a coefficient representation of a function $\mathbf v \in \widetilde{\mathbf W}$. The following Lemma gives an upper bound for that ratio.

\begin{lemma}\label{lem:BdtB:A}
	For all $\mathbf v \in \widetilde{\mathbf W}$ with coefficient representation $\underline{\mathbf v}$, the estimate
	\[
		\| \scaling^{-1} B^\top B \underline{\mathbf v} \|_{S_A}^2
		\lesssim
		\frac{\mu+\lambda}{\mu}\,
		\degree
		\left(1+\log \degree+\max_{k=1,\ldots,K} \log\frac{H_k}{h_k}\right)^2
		\left(\max_{k=1,\ldots,K}\alpha_k^{-2}\right)
		\| \underline{\mathbf v} \|_{S_A}^2
	\]
	holds.
\end{lemma}
\begin{proof}
	Let $\mathbf w,\mathbf v\in \widetilde{\mathbf W}$ with coefficient representations
	$\underline{\mathbf w}$ and $\underline{\mathbf v}$ be such that
	$\underline{\mathbf w} := \scaling^{-1} B^\top B \underline{\mathbf v}$.
	Since the Schur-complement is the energy-minimizing extension and using the boundedness of $A$, we have $\| \scaling^{-1} B^\top B \underline{\mathbf v} \|_{S_A}^2=\| \underline{\mathbf w} \|_{S_A}^2$ with
	\begin{align*}
			\| \underline{\mathbf w} \|_{S_A}^2
			& =
			\sum_{k=1}^K
			\inf_{
			\begin{array}{c}
			\mbox{\scriptsize $\kbox{\mathbf u} \in \kbox{\mathbf V}$}\\
			\mbox{\scriptsize$\left.\kbox{\mathbf u}\right|_{\partial\kbox{\Omega}}= \kbox{\mathbf w}$}
			 \end{array} }
			2\mu \|\epsilon(\kbox{\mathbf u})\|_{L^2(\kbox{\Omega})}^2
			+ \lambda \|\divergence\kbox{\mathbf u}\|_{L^2(\kbox{\Omega})}^2 \\
			&\le
			2(\mu + \lambda)
			\sum_{k=1}^K
				\inf_{
				\begin{array}{c}
				\mbox{\scriptsize $\kbox{\mathbf u} \in \kbox{\mathbf V}$}\\
				\mbox{\scriptsize$\left.\kbox{\mathbf u}\right|_{\partial\kbox{\Omega}}= \kbox{\mathbf w}$}
				 \end{array} }
				|\kbox{\mathbf u}|_{H^1(\kbox{\Omega})}^2.
	\end{align*}
	This term is the $H^1$-energy of the discrete harmonic extension
	$\mathcal H^{(k)}: \mathbf W^{(k)} \rightarrow \mathbf V^{(k)}$,
	which by definition minimizes
	$
		|\mathcal H^{(k)} \mathbf w^{(k)}|_{H^1(\Omega^{(k)})}
	$
	under the constraint
	$
		\left(\mathcal H^{(k)} \mathbf w^{(k)}\right)\Big|_{\partial\Omega^{(k)}}=\mathbf w^{(k)}
	$.
	This means that we have
	\[
			\| \underline{\mathbf w} \|_{S_A}^2
			\le
			2(\mu + \lambda)
			\sum_{k=1}^K
			| \mathcal H^{(k)} \mathbf w^{(k)} |_{H^1(\Omega^{(k)})}^2.
	\]
	Using~\cite[Theorem~4.2 
	and Lemma~4.15]{SchneckenleitnerTakacs:2020},  
	we further have
	\begin{align*}
			\| \underline{\mathbf w} \|_{S_A}^2
			&\lesssim
			2(\mu + \lambda) \degree
			\sum_{k=1}^K
			| \mathbf w^{(k)} |_{H^{1/2}(\partial\Omega^{(k)})}^2
			\\&
			\lesssim
			2(\mu + \lambda) \degree
			\sum_{k=1}^K
			\sum_{\ell\in\mathcal N_\Gamma(k)}
			\left(
			| \mathbf w^{(k)} |_{H^{1/2}(\Gamma^{(k,\ell)})}^2
			+
			\Lambda
			| \mathbf w^{(k)} |_{L_0^\infty(\Gamma^{(k,\ell)})}^2
			\right),
	\end{align*}
	where $\Lambda:=\left(1+\log \degree+\max_{k=1,\ldots,K} \log\frac{H_k}{h_k}\right)$.
	Since $\mathbf v^{(k)}$ is continuous at the corners, standard arguments yield
	\[
		\mathbf w^{(k)}|_{\Gamma^{(k,\ell)}}
		=
		\mathbf v^{(k)}|_{\Gamma^{(k,\ell)}}
		-
		\mathbf v^{(\ell)}|_{\Gamma^{(k,\ell)}},
		\]
	see, e.g.,~\cite[Proof of Lemma~4.16]{SchneckenleitnerTakacs:2020}
	for a proof. Using this statement, we further have
	\[
			\| \underline{\mathbf w} \|_{S_A}^2
			\lesssim
			2(\mu + \lambda) \degree \Lambda
			\sum_{k=1}^K
			\sum_{\ell\in\mathcal N_\Gamma(k)}
			\left(
			| \mathbf v^{(k)} |_{H^{1/2}(\Gamma^{(k,\ell)})}^2
			+
			| \mathbf v^{(k)} |_{L_0^\infty(\Gamma^{(k,\ell)})}^2
			\right).
	\]
	We substitute $\kbox{\mathbf v_0 }:=
	(I-\Pi_F^{(k)}) \kbox{\mathbf v}$ for $\kbox{\mathbf v}$.
	Since $\Pi_F^{(k)} \mathbf v^{(k)}|_{\Gamma^{(k,\ell)}}
	= \Pi_F^{(\ell)} \mathbf v^{(\ell)}|_{\Gamma^{(k,\ell)}}$,
	we know that $\mathbf v-\mathbf v_0$ vanishes on all interfaces. Thus,
	$B \underline{\mathbf  v}_0 = B \underline{\mathbf v} = 0$ and consequently
	\[
			\| \underline{\mathbf w} \|_{S_A}^2
			\lesssim
			\tilde{\Lambda}
			\sum_{k=1}^K
			\sum_{\ell\in\mathcal N_\Gamma(k)}
			\left(
			| (I-\Pi_F^{(k)}) \mathbf v^{(k)} |_{H^{1/2}(\Gamma^{(k,\ell)})}^2
			+
			| (I-\Pi_F^{(k)}) \mathbf v^{(k)} |_{L_0^\infty(\Gamma^{(k,\ell)})}^2
			\right),
	\]
	where $\tilde{\Lambda}:=2(\mu + \lambda) \degree \Lambda$.
	Using \cite[Theorem~4.2]{SchneckenleitnerTakacs:2020} and the fact that the
	$\mathcal H^{(k)}$ is the extension operator that minimizes the $H^1$-seminorm,
	we obtain
	\begin{align*}
			\| \underline{\mathbf w} \|_{S_A}^2
			&\lesssim
			\tilde{\Lambda}
			\sum_{k=1}^K
			\left(
			| \mathcal H^{(k)} (I-\Pi_F^{(k)}) \mathbf v^{(k)} |_{H^{1}(\Omega^{(k)})}^2
			+
			\sum_{\ell\in\mathcal N_\Gamma(k)}
			| (I-\Pi_F^{(k)}) \mathbf v^{(k)} |_{L_0^\infty(\Gamma^{(k,\ell)})}^2
			\right)
			\\
			&\le
			\tilde{\Lambda}
			\sum_{k=1}^K
			\left(
			| \mathcal H_{\epsilon}^{(k)} (I-\Pi_F^{(k)}) \mathbf v^{(k)} |_{H^{1}(\Omega^{(k)})}^2
			+
			\sum_{\ell\in\mathcal N_\Gamma(k)}
			| (I-\Pi_F^{(k)}) \mathbf v^{(k)} |_{L_0^\infty(\Gamma^{(k,\ell)})}^2
			\right)
			,
	\end{align*}
	where $\mathcal H_\epsilon^{(k)}: \mathbf W^{(k)} \rightarrow \mathbf V^{(k)}$,
	which minimizes the energy
	$
			\|\epsilon(\mathcal H_\epsilon^{(k)} \mathbf w^{(k)})\|_{L^2(\Omega^{(k)})}
	$
	under the constraint
	$
		\left(\mathcal H^{(k)}_\epsilon \mathbf w^{(k)}\right)\Big|_{\partial\Omega^{(k)}}=\mathbf w^{(k)}
	$.
	We use the fact that
	$\mathcal H^{(k)}_\epsilon$ does not change the values at the interfaces
	and $\Pi_F^{(k)} \mathcal H^{(k)}_\epsilon =\Pi_F^{(k)}$ to obtain
	\begin{align*}
			\| \underline{\mathbf w} \|_{S_A}^2
			&\lesssim
			\tilde{\Lambda}
			\sum_{k=1}^K
			\left(
			| (\mathcal H_\epsilon^{(k)}-\Pi_F^{(k)}) \mathbf v^{(k)} |_{H^{1}(\Omega^{(k)})}^2
			+
			| (\mathcal H_\epsilon^{(k)}-\Pi_F^{(k)}) \mathbf v^{(k)} |_{L_0^\infty(\Omega^{(k)})}^2
			\right)
			\\
			&\lesssim
			\tilde{\Lambda}
			\sum_{k=1}^K
			\left(
			| (I-\Pi_F^{(k)}) \mathcal H_\epsilon^{(k)}  \mathbf v^{(k)} |_{H^{1}(\Omega^{(k)})}^2
			+
			| (I-\Pi_F^{(k)}) \mathcal H_\epsilon^{(k)}  \mathbf v^{(k)} |_{L_0^\infty(\Omega^{(k)})}^2
			\right).
	\end{align*}
	Since $(I-\Pi_F^{(k)}) \mathbf v^{(k)}$ vanishes at the corners,
	Lemma~\ref{lem:korn2} gives
	\begin{align*}
		&\| \underline{\mathbf w} \|_{S_A}^2 = \| \scaling^{-1} B^\top B \underline{\mathbf v} \|_{S_A}^2
		\\&	\lesssim
			2(\mu + \lambda)\degree \, \Lambda^2 \,\sum_{k=1}^K\alpha_k^{-2}
			\left(
			\| \epsilon( \mathcal H_\epsilon^{(k)}  \mathbf v^{(k)}) \|_{L^2(\Omega^{(k)})}^2
			+
			\| \epsilon( \mathcal H_\epsilon^{(k)} \mathbf v^{(k)}) \|_{L^2(\Omega^{(k)})}^2
			\right)
			\\ & \lesssim
			2(\mu + \lambda)\degree \, \Lambda^2 \, \left(\max_{k=1,\ldots,K}\alpha_k^{-2}\right)
			\sum_{k=1}^K
			\| \epsilon( \mathcal H_\epsilon^{(k)}  \mathbf v^{(k)}) \|_{L^2(\Omega^{(k)})}^2
			\\& \le
			 \frac{\mu + \lambda}{\mu} \degree \, \Lambda^2 \left(\max_{k=1,\ldots,K}\alpha_k^{-2}\right)\, 	\| \underline{\mathbf v} \|_{S_A}^2
			,
	\end{align*}
	which finishes the proof.
\end{proof}

The primal elasticity formulation considered in this paper is an elliptic problem, which directly fits into the framework of ~\cite{MandelDohrmannTezaur:2005a}. Altough, a scalar valued problem was considered there, the main results carry over directly to vector-valued problems. So, we can directly apply~\cite[Theorem~22]{MandelDohrmannTezaur:2005a} to obtain the following  condition number bound.

\begin{theorem}\label{thrm:fin}
	Provided that the IETI-DP solver is set up as outlined in the previous
	section, the condition number of the preconditioned system satisfies
	\[
		\kappaess(M_{\mathrm{sD}} F) \le c\,\frac{\mu+\lambda}{\mu}
		\,
		 \degree \left(1+\log \degree+\max_{k=1,\ldots,K} \log\frac{H_k}{h_k}\right)^2\,
		 \left(\max_{k=1,\ldots,K}\alpha_k^{-2}\right)\,
		 ,
	\]
	where $c$ is a constant that only depends on the constants from the	assumptions~\eqref{eq:ass:neighbor}, \eqref{eq:ass:nabla} and~\eqref{eq:ass:quasiuniform} (cf. Notation~\ref{notation}), $\alpha_k$ are the patch-local Korn constants from~\eqref{eq:localkorn0} and $\kappaess$ denotes the essential condition number, which is the ratio between the largest eigenvalue and the smallest non-zero eigenvalue.
\end{theorem}
\begin{proof}
	\cite[Theorem~22]{MandelDohrmannTezaur:2005a} states that the
	bound follows directly from Lemma~\ref{lem:BdtB:A}.
\end{proof}

If standard Krylov space methods are applied to the singular matrix $F$,
preconditioned with a non-singular preconditioner $M_{\mathrm{sD}}$,
all iterations live in the corresponding factor space. The convergence
behavior is dictated by the essential condition number of
$M_{\mathrm{sD}}F$, cf.~\cite[Remark 23]{MandelDohrmannTezaur:2005a}.

\section{Numerical results}
\label{sec:num}

In this section, we present numerical results that illustrate the efficiency of the proposed IETI-DP solver and the address the issues related to geometry locking. We consider two domains, a tortoise (Subsection~\ref{subsec:GA}) and a wrench (Subsection~\ref{subsec:wrench}). The first domain illustrates that the IETI-DP solver works also well for domains with non-trivial parmeterizations $\mathbf G_k$ and that contain extraordinary vertices (interior vertices touching less than or more than four patches). We consider the wrench domain as a non-trivial example of a long and thin domain, susceptible for geometry locking.

\subsection{Results for the tortoise domain}
\label{subsec:GA}

We consider a tortoise domain, composed of 60 patches, see Figure~\ref{fig:domainGA}. Each of the patches is parameterized using B-splines.
We consider the elasticity problem with Lamé parameters $\mu=\lambda=1$ and without external volumetric forces ($\mathbf f=0$). The boundary conditions are chosen as follows. On the bottom of the feet (depicted in blue color in Figure~\ref{fig:domainGA}), we have homogeneous Dirichlet boundary conditions ($\mathbf g_D=0$). On the top of the tortoise (depicted in red color in Figure~\ref{fig:domainGA}), we assume inhomogeneous Neumann boundary conditions with $\mathbf{g}_N = (0,-g_0)^\top$, with $g_0 = 0.00001$. On the remaining boundary, we set homogeneous Neumann boundary conditions.

\begin{figure}[htb]
	\centering
	\includegraphics[height=14em]{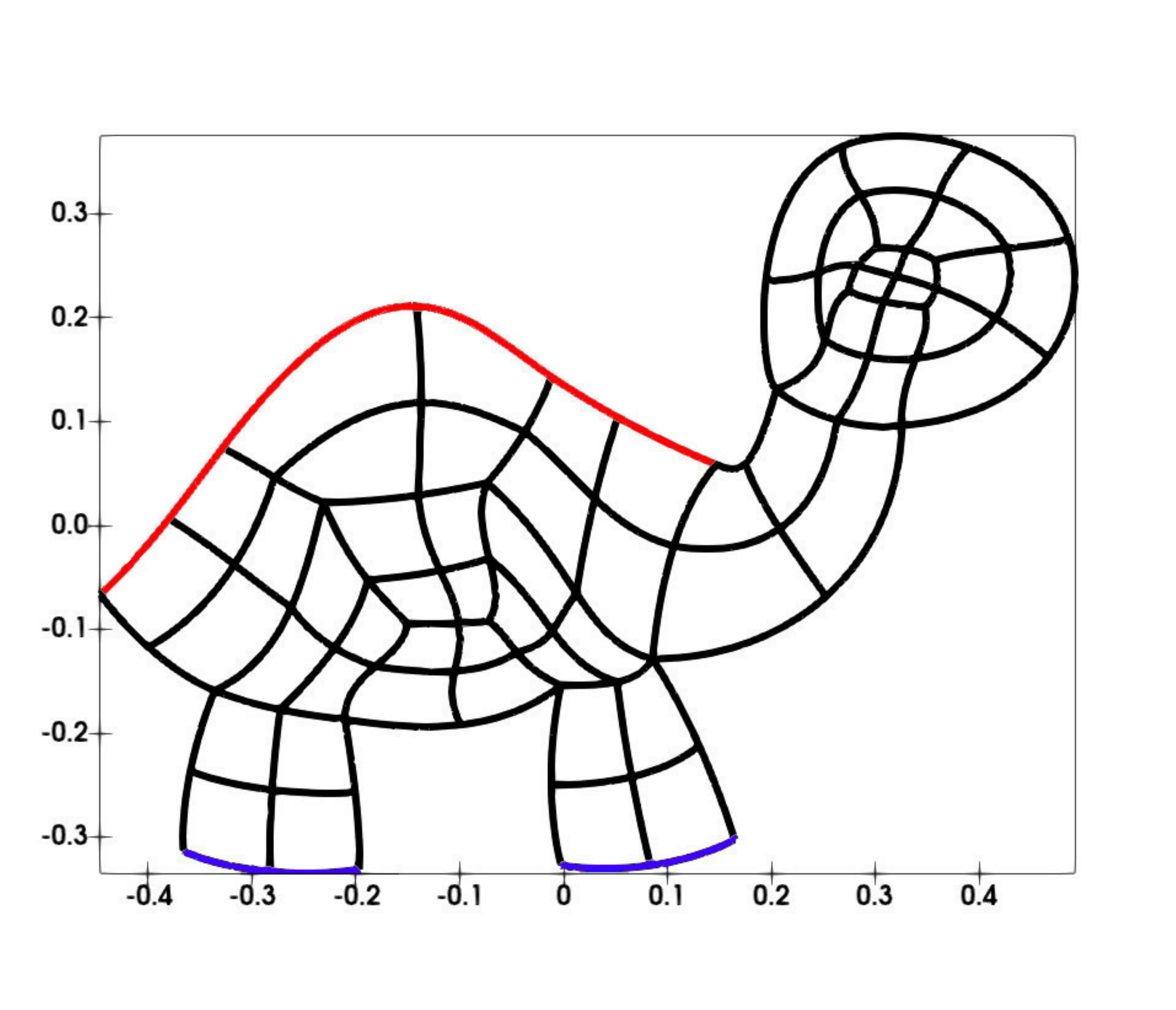}
	\caption{A tortoise}
	\label{fig:domainGA}
\end{figure}

The discretization is done as outlined in Section~\ref{sec:modelProblem}.
On each patch, we obtain the discretization space by performing $\ell = 1, 2,\ldots$ uniform refinement steps, staring from patchwise global polynomials, i.e., a grid without inner knots. For solving the overall problem, we use the IETI-DP method as proposed in Section~\ref{sec:IETISolver}. The reduced system~\eqref{eq:lambdaSys} is solved using a conjugate gradient solver. We use a random initial guess. The overall solver is stopped when the Euclidean norm of the residual is reduced by a factor of $\varepsilon = 10^{-6}$, compared to the Euclidean norm of the initial residual. The condition numbers are estimated using the underlying Lanczos iteration of the conjugate gradient method. The local linear systems are solved using a sparse LU-solver. All experiments have been implemented using the G+Smo library\footnote{\url{https://github.com/gismo/gismo}}.

\begin{table}[th]
\scriptsize
	\newcolumntype{L}[1]{>{\raggedleft\arraybackslash\hspace{-1em}}m{#1}}
	\centering
	\renewcommand{\arraystretch}{1.25}
	\begin{tabular}{l|L{1em}L{1.8em}|L{1em}L{1.8em}|L{1em}L{1.8em}|L{1em}L{1.8em}|L{1em}L{1.8em}}
		\toprule
		& \multicolumn{2}{c|}{$\degree=2$}
		& \multicolumn{2}{c|}{$\degree=3$}
		& \multicolumn{2}{c|}{$\degree=4$}
		& \multicolumn{2}{c|}{$\degree=5$}
		& \multicolumn{2}{c}{$\degree=6$} \\
		$\ell$
		& it & $\kappa$
		& it & $\kappa$
		& it & $\kappa$
		& it & $\kappa$
		& it & $\kappa$ \\
		\midrule
        $2$  & $20$ & $12.2$ & $23$ & $16.6$ & $25$ & $20.0$ & $26$ & $22.6$ & $26$ & $25.2$\\
        $3$  & $24$ & $17.2$ & $27$ & $22.2$ & $29$ & $26.1$ & $30$ & $29.3$ & $31$ & $32.4$\\
        $4$  & $28$ & $22.7$ & $31$ & $28.6$ & $33$ & $33.3$ & $35$ & $37.6$ & $35$ & $42.0$\\
        $5$  & $31$ & $29.2$ & $34$ & $36.6$ & $37$ & $42.8$ & $39$ & $48.8$ & $39$ & $53.6$\\
		\bottomrule
	\end{tabular}
   	\captionof{table}{Iteration counts (it) and condition numbers ($\kappa$); tortoise; $\mu=1, \lambda=1$.
	  \label{tab:GA-A}}
\end{table}
\begin{table}[th]
\scriptsize
	\newcolumntype{L}[1]{>{\raggedleft\arraybackslash\hspace{-1em}}m{#1}}
	\centering
	\renewcommand{\arraystretch}{1.25}
	\begin{tabular}{l|L{1em}L{1.8em}|L{1em}L{1.8em}|L{1em}L{1.8em}|L{1em}L{1.8em}|L{1em}L{1.8em}}
		\toprule
		& \multicolumn{2}{c|}{$\degree=2$}
		& \multicolumn{2}{c|}{$\degree=3$}
		& \multicolumn{2}{c|}{$\degree=4$}
		& \multicolumn{2}{c|}{$\degree=5$}
		& \multicolumn{2}{c}{$\degree=6$} \\
		$\ell$
		& it & $\kappa$
		& it & $\kappa$
		& it & $\kappa$
		& it & $\kappa$
		& it & $\kappa$ \\
		\midrule
        $2$  & $13$ & $ 4.3$ & $16$ & $ 5.9$ & $18$ & $ 7.6$ & $18$ & $ 9.4$ & $19$ & $11.2$\\
        $3$  & $17$ & $ 6.2$ & $20$ & $ 8.7$ & $21$ & $11.4$ & $22$ & $14.0$ & $22$ & $16.2$\\
        $4$  & $20$ & $ 8.7$ & $23$ & $12.7$ & $24$ & $16.4$ & $25$ & $19.5$ & $26$ & $22.2$\\
        $5$  & $23$ & $12.4$ & $26$ & $17.9$ & $27$ & $22.4$ & $29$ & $26.1$ & $29$ & $29.2$\\
        \bottomrule
	\end{tabular}
    \captionof{table}{Iteration counts (it) and condition numbers ($\kappa$); tortoise; $\mu=1, \lambda=1$; including edge averages of normal component of displacement \eqref{eq:normals} as primal degrees of freedom.
	  \label{tab:GA-normal}}
\end{table}
\begin{table}[tb]
\scriptsize
	\newcolumntype{L}[1]{>{\raggedleft\arraybackslash\hspace{-1em}}m{#1}}
	\centering
	\renewcommand{\arraystretch}{1.25}
	\begin{tabular}{l|L{1em}L{1.8em}|L{1em}L{1.8em}|L{1em}L{1.8em}|L{1em}L{1.8em}|L{1em}L{1.8em}}
		\toprule
		& \multicolumn{2}{c|}{$\degree=2$}
		& \multicolumn{2}{c|}{$\degree=3$}
		& \multicolumn{2}{c|}{$\degree=4$}
		& \multicolumn{2}{c|}{$\degree=5$}
		& \multicolumn{2}{c}{$\degree=6$} \\
		$\ell$
		& it & $\kappa$
		& it & $\kappa$
		& it & $\kappa$
		& it & $\kappa$
		& it & $\kappa$ \\
		\midrule
        $2$  & $11$ & $ 2.8$ & $13$ & $ 4.5$ & $15$ & $ 5.9$ & $16$ & $ 7.8$ & $16$ & $ 9.5$\\
        $3$  & $14$ & $ 4.7$ & $16$ & $ 6.8$ & $18$ & $ 9.7$ & $18$ & $12.2$ & $18$ & $14.3$\\
        $4$  & $17$ & $ 6.6$ & $19$ & $10.8$ & $20$ & $14.4$ & $21$ & $17.4$ & $21$ & $19.9$\\
        $5$  & $19$ & $10.2$ & $21$ & $15.7$ & $24$ & $20.1$ & $24$ & $23.6$ & $24$ & $26.5$\\
        \bottomrule
	\end{tabular}
    \captionof{table}{Iteration counts (it) and condition numbers ($\kappa$); tortoise; $\mu=1, \lambda=1$; including edge averages of both displacement components \eqref{eq:edge} as primal degrees of freedom.
	  \label{tab:GA-edge}}
\end{table}

For this setting, we obtain a moderate global Korn constant $\alpha_{\Omega,\Gamma_D}^{-2} \approx 211$; here and in what follows, we have estimated the constant for $\degree=3$, $\ell=2$ using a singular value decomposition.

The number of iterations needed to reach the threshold and the estimated condition number are shown in Table~\ref{tab:GA-A}. As proposed, we only considered the corner values as primal degrees of freedom. From the tables we see that the iteration count and condition number increase slightly as $\ell$ and $\degree$ increases. This is consistent with the convergence estimate in Theorem~\ref{thrm:fin}.

In Table~\ref{tab:GA-normal}, we present the results, where we enrich the primal degrees of freedom.  We follow the approach presented in~\cite{SognTakacs:2022} for the Stokes equations. Here, besides the corner values, we also use the averages of the normal components as primal degrees of freedom, i.e., we require
\begin{equation}\label{eq:normals}
		\int_{\Gamma^{(k,\ell)}} \mathbf u^{(k)} \cdot \mathbf n^{(k)} \,\mathrm d s
		+
		\int_{\Gamma^{(k,\ell)}} \mathbf u^{(\ell)} \cdot \mathbf n^{(\ell)} \,\mathrm d s
		=0
\end{equation}
for any two patches $\Omega^{(k)}$ and $\Omega^{(\ell)}$ that share a common edge $\Gamma^{(k,\ell)}:= \partial\Omega^{(k)}\cap\partial\Omega^{(\ell)}$. We observe that this choice significantly reduces the iteration count and condition number.

In Table~\ref{tab:GA-edge}, we present the results, where we enrich the primal degrees of freedom using standard edge averages
\begin{equation}\label{eq:edge}
		\int_{\Gamma^{(k,\ell)}} \mathbf u^{(k)}  \,\mathrm d s
		=
		\int_{\Gamma^{(k,\ell)}} \mathbf u^{(\ell)}  \,\mathrm d s,
\end{equation}
which we use besides the corner values. This option yields slightly better results than the results with normal components only, however this approach leads to an increase of the number of primal degrees of freedom by a factor of around $\tfrac32$, compared to the version presented in Table~\ref{tab:GA-normal}.

\subsection{Results for a wrench}
\label{subsec:wrench}

In this Subsection, we demonstrate that the convergence of the proposed method does not suffer in cases where the constant from Korn's inequality is small. As computational domain, we choose a wrench with a long and relatively thin handle, see Figure~\ref{fig:domainW}. The patch is compost of $64$ patches.
\begin{figure}[htb]
	\centering
	\includegraphics[height=3cm]{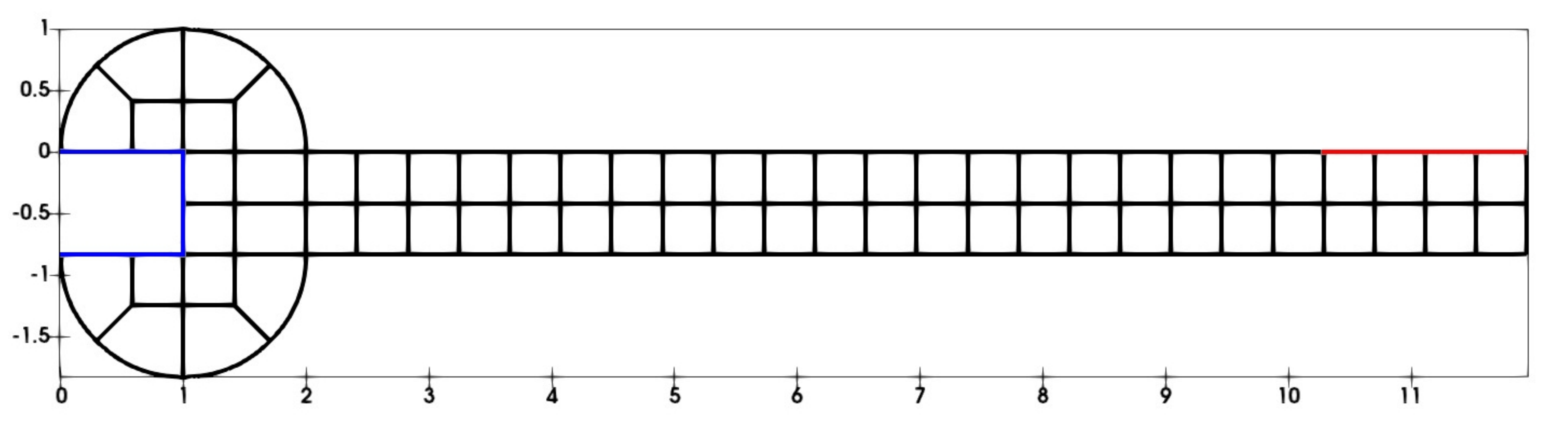}
	\caption{A wrench}
	\label{fig:domainW}
\end{figure}
Again, we consider the elasticity problem with Lamé parameters $\mu=\lambda=1$ and with $\mu=1,\lambda=100$ and without external volumetric forces ($\mathbf f=0$). We choose a homogeneous Dirichlet boundary condition on the sides where the wrench would touch the bolt, this is marked with blue color in Figure~\ref{fig:domainW}. The remaining boundary has Neumann boundary conditions. For the parts that are marked with red color, we impose a downward force, represented as $\mathbf{g}_N = (0,-g_0)$, with $g_0 = 0.00001$. For the remaining boundary, we have homogeneous Neumann boundary conditions.

The discretization and solution strategy is the same as for the tortoise (Subsection~\ref{subsec:GA}). Compared to the tortoise domain, the estimated value of the constant for Korn's inequality is much smaller, since we now get $\alpha_{\Omega,\Gamma_D}^{-2}\approx 1500$.

The iteration count and the estimated condition number are shown in Table~\ref{tab:wrench} and \ref{tab:wrenchNormal}.
For Table~\ref{tab:wrenchNormal}, 'normal' dofs where also used as primals.
The global Korn constant is much smaller for the wrench, however, the iteration counts and condition numbers are better than the tortoise. This indicates, as predicted from the theory, that the global Korn constant does not effect the condition number of the IETI-DP solver.  The reason the tortoise has larger iteration counts and condition numbers, is because the individual geometry mappings are more distorted than for the wrench.

\begin{table}[b]
\scriptsize
	\newcolumntype{L}[1]{>{\raggedleft\arraybackslash\hspace{-1em}}m{#1}}
	\centering
	\renewcommand{\arraystretch}{1.25}
	\begin{tabular}{l|L{1em}L{1.8em}|L{1em}L{1.8em}|L{1em}L{1.8em}|L{1em}L{1.8em}|L{1em}L{1.8em}|L{1em}L{1.8em}}
		\toprule
		& \multicolumn{2}{c|}{$\degree=1$}
		& \multicolumn{2}{c|}{$\degree=2$}
		& \multicolumn{2}{c|}{$\degree=3$}
		& \multicolumn{2}{c|}{$\degree=4$}
		& \multicolumn{2}{c|}{$\degree=5$}
		& \multicolumn{2}{c}{$\degree=6$} \\
		$\ell$
		& it & $\kappa$
		& it & $\kappa$
		& it & $\kappa$
		& it & $\kappa$
		& it & $\kappa$ \\
		\midrule
        $2$  & $13$ & $ 3.8$ & $17$ & $ 6.0$ & $19$ & $ 7.4$ & $20$ & $ 8.4$ & $20$ & $ 9.4$ & $20$ & $10.0$\\
        $3$  & $16$ & $ 5.4$ & $19$ & $ 7.9$ & $21$ & $ 9.4$ & $22$ & $10.6$ & $23$ & $11.5$ & $23$ & $12.4$\\
 		$4$  & $18$ & $ 7.2$ & $22$ & $10.0$ & $24$ & $11.7$ & $25$ & $13.1$ & $26$ & $14.3$ & $26$ & $15.3$\\
        $5$  & $21$ & $ 9.2$ & $24$ & $12.5$ & $26$ & $14.5$ & $28$ & $16.1$ & $29$ & $17.4$ & $29$ & $18.6$\\
		\bottomrule
	\end{tabular}
    \captionof{table}{Iteration count (it) and condition numbers ($\kappa$); wrench; $\mu=1, \lambda=1$.
	  \label{tab:wrench}}
\end{table}
\begin{table}[thb]
\scriptsize
	\newcolumntype{L}[1]{>{\raggedleft\arraybackslash\hspace{-1em}}m{#1}}
	\centering
	\renewcommand{\arraystretch}{1.25}
	\begin{tabular}{l|L{1em}L{1.8em}|L{1em}L{1.8em}|L{1em}L{1.8em}|L{1em}L{1.8em}|L{1em}L{1.8em}|L{1em}L{1.8em}}
		\toprule
		& \multicolumn{2}{c|}{$\degree=1$}
		& \multicolumn{2}{c|}{$\degree=2$}
		& \multicolumn{2}{c|}{$\degree=3$}
		& \multicolumn{2}{c|}{$\degree=4$}
		& \multicolumn{2}{c|}{$\degree=5$}
		& \multicolumn{2}{c}{$\degree=6$} \\
		$\ell$
		& it & $\kappa$
		& it & $\kappa$
		& it & $\kappa$
		& it & $\kappa$
		& it & $\kappa$ \\
		\midrule
        $2$  & $ 8$ & $ 2.0$ & $11$ & $ 3.0$ & $13$ & $ 3.8$ & $14$ & $ 4.6$ & $14$ & $ 5.1$ & $15$ & $ 5.7$ \\
        $3$  & $11$ & $ 2.8$ & $14$ & $ 4.2$ & $15$ & $ 5.2$ & $16$ & $ 6.1$ & $17$ & $ 6.7$ & $17$ & $ 7.4$ \\
 		$4$  & $13$ & $ 3.8$ & $16$ & $ 5.6$ & $18$ & $ 6.9$ & $19$ & $ 7.8$ & $20$ & $ 8.6$ & $20$ & $ 9.3$ \\
        $5$  & $15$ & $ 5.2$ & $18$ & $ 7.3$ & $20$ & $ 8.8$ & $22$ & $ 9.9$ & $22$ & $10.8$ & $22$ & $11.5$ \\
		\bottomrule
	\end{tabular}
    \captionof{table}{Iteration count (it) and condition numbers ($\kappa$); wrench; $\mu=1, \lambda=1$; including edge averages of normal component of displacement \eqref{eq:normals} as primal degrees of freedom.
	  \label{tab:wrenchNormal}}
\end{table}

\begin{table}[htp]
\scriptsize
	\newcolumntype{L}[1]{>{\raggedleft\arraybackslash\hspace{-1em}}m{#1}}
	\centering
	\renewcommand{\arraystretch}{1.25}
	\begin{tabular}{l|L{1em}L{1.8em}|L{1em}L{1.8em}|L{1em}L{1.8em}|L{1em}L{1.8em}|L{1em}L{1.8em}|L{1em}L{1.8em}}
		\toprule
		& \multicolumn{2}{c|}{$\degree=1$}
		& \multicolumn{2}{c|}{$\degree=2$}
		& \multicolumn{2}{c|}{$\degree=3$}
		& \multicolumn{2}{c|}{$\degree=4$}
		& \multicolumn{2}{c|}{$\degree=5$}
		& \multicolumn{2}{c}{$\degree=6$} \\
		$\ell$
		& it & $\kappa$
		& it & $\kappa$
		& it & $\kappa$
		& it & $\kappa$
		& it & $\kappa$ \\
		\midrule
        $2$  & $19$ & $20.5$ & $28$ & $50.6$  & $34$ & $67.7$  & $37$ & $78.4$  & $37$ & $87.2$  & $37$ & $95.1$\\
        $3$  & $22$ & $33.5$ & $31$ & $66.8$  & $36$ & $84.0$  & $39$ & $95.7$  & $41$ & $104.8$ & $41$ & $110.6$\\
 		$4$  & $26$ & $48.7$ & $34$ & $83.9$  & $39$ & $101.0$ & $42$ & $111.8$ & $43$ & $119.8$ & $45$ & $129.2$\\
        $5$  & $30$ & $64.5$ & $37$ & $100.8$ & $42$ & $119.1$ & $45$ & $130.6$ & $46$ & $141.8$ & $47$ & $147.7$\\
		\bottomrule
	\end{tabular}
    \captionof{table}{Iteration count (it) and condition numbers ($\kappa$); wrench; $\mu=1, \lambda=100$.
	  \label{tab:wrench100}}
\end{table}
\begin{table}[htp]
\scriptsize
	\newcolumntype{L}[1]{>{\raggedleft\arraybackslash\hspace{-1em}}m{#1}}
	\centering
	\renewcommand{\arraystretch}{1.25}
	\begin{tabular}{l|L{1em}L{1.8em}|L{1em}L{1.8em}|L{1em}L{1.8em}|L{1em}L{1.8em}|L{1em}L{1.8em}|L{1em}L{1.8em}}
		\toprule
		& \multicolumn{2}{c|}{$\degree=1$}
		& \multicolumn{2}{c|}{$\degree=2$}
		& \multicolumn{2}{c|}{$\degree=3$}
		& \multicolumn{2}{c|}{$\degree=4$}
		& \multicolumn{2}{c|}{$\degree=5$}
		& \multicolumn{2}{c}{$\degree=6$} \\
		$\ell$
		& it & $\kappa$
		& it & $\kappa$
		& it & $\kappa$
		& it & $\kappa$
		& it & $\kappa$ \\
		\midrule
        $2$  & $10$ & $2.9$ & $16$ & $ 6.8$ & $20$ & $10.6$ & $22$ & $14.0$ & $23$ & $16.3$ & $23$ & $17.9$\\
        $3$  & $12$ & $3.6$ & $17$ & $ 8.1$ & $21$ & $12.3$ & $23$ & $15.3$ & $24$ & $17.6$ & $24$ & $19.4$\\
 		$4$  & $13$ & $4.4$ & $19$ & $ 9.1$ & $22$ & $13.4$ & $24$ & $16.5$ & $25$ & $18.7$ & $26$ & $20.4$\\
        $5$  & $14$ & $5.0$ & $20$ & $10.0$ & $23$ & $14.2$ & $25$ & $17.3$ & $26$ & $19.6$ & $27$ & $21.2$\\
		\bottomrule
	\end{tabular}
    \captionof{table}{Iteration count (it) and condition numbers ($\kappa$); wrench; $\mu=1, \lambda=100$; including edge averages of normal component of displacement \eqref{eq:normals} as primal degrees of freedom.
	  \label{tab:wrenchNormal100}}
\end{table}

Next, we demonstrate that the IETI-DP solver also works well if $\lambda/\mu$ is increased, which corresponds to the case that the material is less compressible. The results are presented in Tables~\ref{tab:wrench100} and \ref{tab:wrenchNormal100}. We observe that the iteration counts only mildly grow compared to the case $\lambda=\mu=1$. The condition numbers seem to be much more susceptible to this change. All effects are reduced if, additionally to the corner values, also the normal components of the edge averages are chosen as primal degrees of freedom.

In Table~\ref{tab:wrenchsize}, we consider the dependence of the performance of the IETI-DP solver on the length of the wrench. The wrench depicted in Figure~\ref{fig:domainW} has handle with a length of $10$ (with $x$-coordinate between $2$ and $12$). The other choices have half, double or quadruple the length. For the first three rows, we estimated Korn's constant using a singular value decomposition for $\degree=3$ and $\ell=2$ ($\ell=1$ for $K=112$). For $K=208$, the (global) singular value decomposition was not feasible. Again, the homogeneous Neumann boundary conditions are imposed on the top sides of the 4 patches on the top-right side of the handle. Although the global Korn constant deteriorates with the length of the handle, the iteration counts and condition number estimates show that the solver does not suffer.

Finally, we also aim to illustrate on the discretization error. Namely, in Figure~\ref{fig:Bending}, we present the total bending of the wrench (with $K=64$), i.e., we consider the value of the $y$-component at the lower right corner of the wrench.  Here, the linear system is solved up to an accuracy of $\varepsilon = 10^{-12}$.  In the figures, we present the number of refinement levels on the $x$-axis, the total bending on the $y$-axes. In all configurations the total bending seems to converge for $\ell\to\infty$ (which is to be expected since we have shown a discretization error estimate $\|\mathbf u -\mathbf u_h\|_{H^1(\Omega)}\to0$ for $h\to0$). We observe that the total bending for small $\ell$ and $\degree$ is too small, which means that the material appears as too stiff. This is what one expects in case of geometry locking. In all cases, we observe convergence for $\ell\to \infty$. The convergence is comparably slow for $\degree=1$. Here, the relative error (compared to the total bending as computed for $\ell=5,\degree=6$) is around $5\,\%$ for $\lambda=1$ and $62\,\%$ for $\lambda=100$. For $\degree>1$, already the coarsest grid solution ($\ell=1$) gives rather good approximations for the total bending; the error is below $1\,\%$ for $\lambda=1$ and below $3\,\%$ for $\lambda=100$. This reflects the observations from Section~\ref{sec:approx}.

\begin{table}[b]
\scriptsize
	\newcolumntype{L}[1]{>{\raggedleft\arraybackslash\hspace{-1em}}m{#1}}
	\centering
	\renewcommand{\arraystretch}{1.25}
	\begin{tabular}{lll|L{1em}L{1.8em}|L{1em}L{1.8em}|L{1em}L{1.8em}|L{1em}L{1.8em}|L{1em}L{1.8em}|L{1em}L{1.8em}}
		\toprule
		&&
		& \multicolumn{2}{c|}{$\degree=1$}
		& \multicolumn{2}{c|}{$\degree=2$}
		& \multicolumn{2}{c|}{$\degree=3$}
		& \multicolumn{2}{c|}{$\degree=4$}
		& \multicolumn{2}{c|}{$\degree=5$}
		& \multicolumn{2}{c}{$\degree=6$} \\
		length &K & $\alpha_{\Omega,\Gamma_D}^{-2} $
		& it & $\kappa$
		& it & $\kappa$
		& it & $\kappa$
		& it & $\kappa$
		& it & $\kappa$ \\
		\midrule
        $5$  & $40$ & $ \approx402 $
        & $16$ & $5.4$ & $19$ & $ 7.9$ & $21$ & $ 9.4$ & $22$ & $10.5$ & $22$ & $11.5$ & $23$ & $12.4$\\
        $10$ & $64$  & $\approx1500$
        & $16$ & $5.4$ & $19$ & $ 7.9$ & $21$ & $9.4$ & $22$ & $ 10.6$ & $23$ & $11.5$ & $23$ & $12.4$\\
 		$20$ & $112$ & $\approx5794$
 		& $16$ & $5.4$ & $20$ & $ 7.9$ & $21$ & $ 9.4$ & $22$ & $10.6$ & $23$ & $11.5$ & $23$ & $12.4$ \\
        $40$  & $208$ & $-$
        & $16$ & $5.4$ & $19$ & $ 7.9$ & $21$ & $ 9.4$ & $23$ & $10.6$ & $23$ & $11.6$ & $23$ & $12.5$\\
		\bottomrule
	\end{tabular}
    \captionof{table}{Iteration count (it) and condition numbers ($\kappa$) for the wrench; $\mu=1, \lambda=1, \ell=3$
	  \label{tab:wrenchsize}}
\end{table}

\begin{figure}[htp]
\centering
\scalebox{.7}{
\begin{tikzpicture}
\begin{axis}[
    xlabel={$\ell$},
    xmin=0.5, xmax=5.5,
    ymin=-0.043, ymax=-0.039,
    xtick={1,2,3,4,5},
    ytick={-0.043,-0.042,-0.041,-0.040},
    legend pos=north east,
    ymajorgrids=true,
    grid style=dashed,
]

	\addplot[
    color=blue,
    mark=square,
    ]
    coordinates {
    (1,-0.0400272)
    (2,-0.0413810)
    (3,-0.0418614)
    (4,-0.0420425)
    (5,-0.0421153)
    };

	\addplot[
    color=red,
    mark=o,
    ]
    coordinates {
    (1,-0.0418145)
    (2,-0.0420007)
    (3,-0.0420915)
    (4,-0.0421340)
    (5,-0.0421538)
    };

	\addplot[
    color=green,
    mark=*,
    ]
    coordinates {
    (1,-0.0419759)
    (2,-0.0420730)
    (3,-0.0421250)
    (4,-0.0421495)
    (5,-0.0421610)
    };

	\addplot[
    color=black,
    mark=x,
    ]
    coordinates {
    (1,-0.0421082)
    (2,-0.0421347)
    (3,-0.0421533)
    (4,-0.0421627)
    (5,-0.0421671)
    };

    \legend{$\degree=1$,$\degree=2$,$\degree=3$,$\degree=6$}

\end{axis}
\end{tikzpicture}
}\quad\scalebox{.7}{
\begin{tikzpicture}
\begin{axis}[
    xlabel={$\ell$},
    xmin=0.5, xmax=5.5,
    ymin=-0.03, ymax=-0.01,
    xtick={1,2,3,4,5},
    ytick={-0.01, -0.015, -0.02, -0.025, -0.03},
    legend pos=north east,
    ymajorgrids=true,
    grid style=dashed,
]

	\addplot[
    color=blue,
    mark=square,
    ]
    coordinates {
    (1,-0.0109634)
    (2,-0.0198972)
    (3,-0.0253928)
    (4,-0.0274692)
    (5,-0.0281453)
    };

	\addplot[
    color=red,
    mark=o,
    ]
    coordinates {
    (1,-0.0277644)
    (2,-0.0281289)
    (3,-0.0283260)
    (4,-0.0284215)
    (5,-0.0284665)
    };

	\addplot[
    color=green,
    mark=*,
    ]
    coordinates {
    (1,-0.0281591)
    (2,-0.0283243)
    (3,-0.0284198)
    (4,-0.0284656)
    (5,-0.0284870)
    };

	\addplot[
    color=black,
    mark=x,
    ]
    coordinates {
    (1,-0.0284117)
    (2,-0.0284509)
    (3,-0.0284788)
    (4,-0.0284931)
    (5,-0.0284998)
    };
    \legend{$\degree=1$,$\degree=2$,$\degree=3$,$\degree=6$}

\end{axis}
\end{tikzpicture}
}
    \captionof{figure}{Total bending for the wrench; $\mu=1, \lambda=1$ (left) and
    $\lambda=100$ (right).
	  \label{fig:Bending}}
\end{figure}
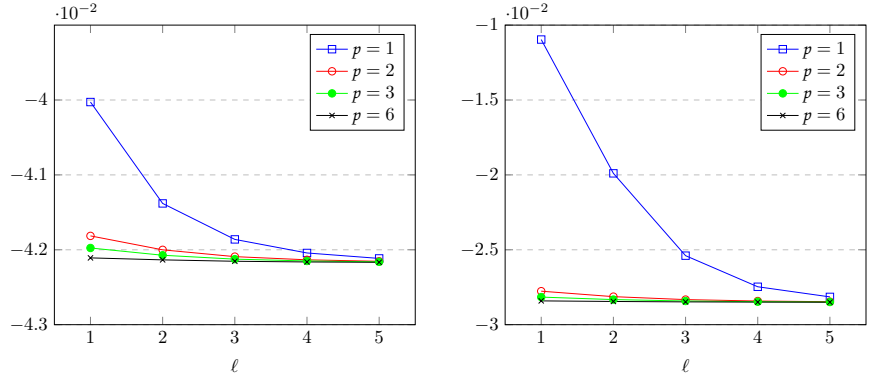

\subsection{Concluding remarks}
In this paper, a IETI-DP solution strategy was proposed and a bound for the condition number of the preconditioned system was proven. This bound is explicit with respect to the grid size, the spline degree, and the local Korn constants. This means that if the patching of the domain is done well, the solution strategy will not suffer from geometry locking. The numerical experiments agree with our theoretical findings. Material locking, which occurs when $\lambda \gg \mu$, was not addressed in this paper. This will be addressed in a forthcoming paper.

\section*{Acknowledgments}
This work was supported by the Austrian Science Fund (FWF): P31048. The first author acknowledges support from the Research Council of Norway, grant 301013. This support is gratefully acknowledged.


\begin{thebibliography}{00}

\bibitem{AcostaDuran:2017}
G.~Acosta and R.~G. Dur{\'a}n.
\newblock {\em Divergence operator and related inequalities}.
\newblock Springer, 2017.

\bibitem{arnold2007mixed}
D.~Arnold, R.~Falk, and R.~Winther.
\newblock Mixed finite element methods for linear elasticity with weakly
  imposed symmetry.
\newblock {\em Mathematics of Computation}, 76(260):1699 -- 1723, 2007.

\bibitem{auricchio2007fully}
F.~Auricchio, L.~Beirão~da Veiga, A.~Buffa, C.~Lovadina, A.~Reali, and
  G.~Sangalli.
\newblock A fully “locking-free” isogeometric approach for plane linear
  elasticity problems: {A} stream function formulation.
\newblock {\em Computer methods in applied mechanics and engineering},
  197(1-4):160 -- 172, 2007.

\bibitem{da2014mathematical}
L.~Beirão~da Veiga, A.~Buffa, G.~Sangalli, and R.~V{\'a}zquez.
\newblock Mathematical analysis of variational isogeometric methods.
\newblock {\em Acta Numerica}, 23:157 -- 287, 2014.

\bibitem{boffi2013mixed}
D.~Boffi, F.~Brezzi, M.~Fortin, et~al.
\newblock {\em Mixed finite element methods and applications}, volume~44.
\newblock Springer, 2013.

\bibitem{BrennerScott}
S.~Brenner and L.~R. Scott.
\newblock {\em {The Mathematical Theory of Finite Element Methods (Third
  edition)}}.
\newblock Texts in Applied Mathematics. Springer, 2008.

\bibitem{Cottrell:Hughes:Bazilevs}
J.~A. Cottrell, T.~J.~R. Hughes, and Y.~Bazilevs.
\newblock {\em Isogeometric Analysis -- Toward Integration of CAD and FEA}.
\newblock John Wiley \& Sons, 2009.

\bibitem{farhat2001feti}
C.~Farhat, M.~Lesoinne, P.~LeTallec, K.~Pierson, and D.~Rixen.
\newblock {FETI-DP: a dual--primal unified FETI method—part I: A faster
  alternative to the two-level FETI method}.
\newblock {\em International journal for numerical methods in engineering},
  50(7):1523 -- 1544, 2001.

\bibitem{farhat2000scalable}
C.~Farhat, M.~Lesoinne, and K.~Pierson.
\newblock A scalable dual-primal domain decomposition method.
\newblock {\em Numerical linear algebra with applications}, 7(7-8):687 -- 714,
  2000.

\bibitem{gould1994introduction}
P.~L. Gould and Y.~Feng.
\newblock {\em Introduction to linear elasticity}.
\newblock Springer, 1994.

\bibitem{hofer2017dual}
C.~Hofer and U.~Langer.
\newblock {Dual-primal isogeometric tearing and interconnecting solvers for
  multipatch dG-IgA equations}.
\newblock {\em Computer Methods in Applied Mechanics and Engineering}, 316:2 --
  21, 2017.

\bibitem{hughes2005isogeometric}
T.~J. Hughes, J.~A. Cottrell, and Y.~Bazilevs.
\newblock {Isogeometric analysis: CAD, finite elements, NURBS, exact geometry
  and mesh refinement}.
\newblock {\em Computer methods in applied mechanics and engineering},
  194(39-41):4135 -- 4195, 2005.

\bibitem{KLAWONN20071400}
A.~Klawonn and O.~Rheinbach.
\newblock Robust {FETI-DP} methods for heterogeneous three dimensional
  elasticity problems.
\newblock {\em Computer Methods in Applied Mechanics and Engineering},
  196(8):1400 -- 1414, 2007.
\newblock Domain Decomposition Methods: recent advances and new challenges in
  engineering.

\bibitem{KlawonnWidlund:2006}
A.~Klawonn and O.~Widlund.
\newblock Dual‐primal {FETI} methods for linear elasticity.
\newblock {\em Communications on Pure and Applied Mathematics}, 59:1523 --
  1572, 11 2006.

\bibitem{kleiss2012ieti}
S.~K. Kleiss, C.~Pechstein, B.~J{\"u}ttler, and S.~Tomar.
\newblock {IETI}--isogeometric tearing and interconnecting.
\newblock {\em Computer Methods in Applied Mechanics and Engineering}, 247:201
  -- 215, 2012.

\bibitem{MandelDohrmannTezaur:2005a}
J.~Mandel, C.~R. Dohrmann, and R.~Tezaur.
\newblock An algebraic theory for primal and dual substructuring methods by
  constraints.
\newblock {\em Appl. Numer. Math.}, 54(2):167 -- 193, 2005.

\bibitem{montardini2022ieti}
M.~Montardini, G.~Sangalli, R.~Schneckenleitner, S.~Takacs, and M.~Tani.
\newblock {A IETI-DP method for discontinuous Galerkin discretizations in
  Isogeometric Analysis with inexact local solvers}.
\newblock Technical report, 2022.
\newblock (Available as arXiv preprint arXiv:2206.08416).

\bibitem{pavarino2016isogeometric}
L.~F. Pavarino and S.~Scacchi.
\newblock {Isogeometric block FETI-DP preconditioners for the Stokes and mixed
  linear elasticity systems}.
\newblock {\em Computer Methods in Applied Mechanics and Engineering}, 310:694
  -- 710, 2016.

\bibitem{SandeManniSpeleers:2019}
E.~Sande, C.~Manni, and H.~Speleers.
\newblock {Sharp error estimates for spline approximation: explicit constants,
  $n$-widths, and eigenfunction convergence}.
\newblock {\em Math. Models Methods Appl. Sci.}, 29(6):1175 --– 1205, 2019.

\bibitem{SchneckenleitnerTakacs:2020}
R.~Schneckenleitner and S.~Takacs.
\newblock Condition number bounds for {IETI-DP} methods that are explicit in
  $h$ and $p$.
\newblock {\em Mathematical Models and Methods in Applied Sciences},
  30(11):2067 -- 2103, 2020.

\bibitem{schneckenleitner2022ieti}
R.~Schneckenleitner and S.~Takacs.
\newblock {IETI-DP methods for discontinuous Galerkin multi-patch Isogeometric
  Analysis with T-junctions}.
\newblock {\em Computer Methods in Applied Mechanics and Engineering},
  393:114694, 2022.

\bibitem{schneckenleitner2021inexact}
R.~Schneckenleitner and S.~Takacs.
\newblock {Inexact IETI-DP for conforming isogeometric multi-patch
  discretizations}.
\newblock In {\em Spectral and High Order Methods for Partial Differential
  Equations ICOSAHOM 2021}, 2023.
\newblock To appear. (Also available as arXiv preprint: arXiv:2110.06087).

\bibitem{SognTakacs:2022}
J.~Sogn and S.~Takacs.
\newblock {Stable discretizations and IETI-DP solvers for the Stokes system in
  multi-patch Isogeometric Analysis}.
\newblock {\em ESAIM: Mathematical Modelling and Numerical Analysis}, 57(2):921
  --– 925, 2022.

\bibitem{stenberg1986construction}
R.~Stenberg.
\newblock On the construction of optimal mixed finite element methods for the
  linear elasticity problem.
\newblock {\em Numerische Mathematik}, 48(4):447 -- 462, 1986.

\bibitem{takacs2018robust}
S.~Takacs.
\newblock {Robust approximation error estimates and multigrid solvers for
  isogeometric multi-patch discretizations}.
\newblock {\em Mathematical Models and Methods in Applied Sciences},
  28(10):1899 -- 1928, 2018.

\bibitem{TakacsTakacs:2016}
S.~Takacs and T.~Takacs.
\newblock {Approximation error estimates and inverse inequalities for B-splines
  of maximum smoothness}.
\newblock {\em Math. Models Methods Appl. Sci.}, 26(7):1411 -- 1445, 2016.

\bibitem{widlund2021block}
O.~Widlund, S.~Zampini, S.~Scacchi, and L.~F. Pavarino.
\newblock {Block FETI--DP/BDDC preconditioners for mixed isogeometric
  discretizations of three-dimensional almost incompressible elasticity}.
\newblock {\em Mathematics of Computation}, 90(330):1773 -- 1797, 2021.

\end{thebibliography}
\end{document}